\documentclass[a4paper, 11pt, reqno]{amsart}
\usepackage[left=1in, right=1in, top=1.2in, bottom=1.2in]{geometry}

\usepackage{times}
\usepackage{microtype}

\usepackage{commath, amssymb, amscd, mathrsfs, mathtools, dsfont, bbm}
\usepackage[all]{xy}
\usepackage{enumerate}

\usepackage{cite}       

\usepackage[usenames,dvipsnames,svgnames,table]{xcolor}
\definecolor{red}{RGB}{255,25,25}
\definecolor{blue}{RGB}{25,50,200}

\usepackage{fancyhdr}
\fancypagestyle{appendix}{
\setlength{\headheight}{10pt}
\fancyfoot{}
\fancyhead[EC]{\normalfont\scriptsize TOMOHIDE TERASOMA}
\fancyhead[EL]{\normalfont\scriptsize\thepage}
\fancyhead[OC]{\normalfont\scriptsize AN EXAMPLE OF A SOLVABLE SUBGROUP OF AN AUTOMORPHISM GROUP}
\fancyhead[OR]{\normalfont\scriptsize\thepage}
\fancyhead[ER]{}
\fancyhead[OL]{}

}

\def\MR#1{\href{https://mathscinet.ams.org/mathscinet-getitem?mr=#1}{MR#1}}
\providecommand{\bysame}{\leavevmode\hbox to3em{\hrulefill}\thinspace}
\providecommand{\MR}{\relax\ifhmode\unskip\space\fi MR }

\providecommand{\href}[2]{#2}

\usepackage[pagebackref, linktocpage]{hyperref}
\hypersetup{
    pdftitle={Tits alt in char p},            
    pdfauthor={Fei Hu},     
    colorlinks=true,        
    linkcolor=red,          
    citecolor=MidnightBlue,         
    filecolor=magenta,      
    urlcolor=MidnightBlue           
}

\newtheorem{theorem}{Theorem}[section]
\newtheorem{lemma}[theorem]{Lemma}
\newtheorem{proposition}[theorem]{Proposition}
\newtheorem{corollary}[theorem]{Corollary}

\theoremstyle{definition}
\newtheorem{definition}[theorem]{Definition}

\newtheorem{example}[theorem]{Example}

\newtheorem*{convention}{Convention}

\theoremstyle{remark}
\newtheorem{remark}[theorem]{Remark}
\newtheorem*{rmk}{Remark}

\newtheorem{step}{\sc Step}

\numberwithin{equation}{section}

\newcommand{\eps}{\epsilon}

\newcommand{\wbar}[1]{\overline{#1}}
\newcommand{\ol}[1]{\overline{#1}}

\newcommand{\lra}{\longrightarrow}
\newcommand{\ratmap}{\dashrightarrow}
\newcommand{\injmap}{\hookrightarrow}
\newcommand{\longinjmap}{\lhook\joinrel\longrightarrow}
\newcommand{\surjmap}{\twoheadrightarrow}
\newcommand{\longsurjmap}{\relbar\joinrel\twoheadrightarrow}
\newcommand{\arxiv}[1]{\href{https://arxiv.org/abs/#1}{{\tt arXiv:#1}}}

\newcommand{\bbK}{\mathbb{K}}

\newcommand{\bk}{\mathbbm{k}}

\newcommand{\bH}{\mathbf{H}}
\newcommand{\bK}{\mathbf{K}}

\newcommand{\bP}{\mathbf{P}}

\newcommand{\bC}{\mathbb{C}}
\newcommand{\bF}{\mathbb{F}}

\newcommand{\bQ}{\mathbb{Q}}
\newcommand{\bR}{\mathbb{R}}
\newcommand{\bZ}{\mathbb{Z}}

\newcommand{\sE}{\mathscr{E}}

\newcommand{\sI}{\mathscr{I}}

\newcommand{\sL}{\mathscr{L}}

\newcommand{\sO}{\mathscr{O}}

\newcommand{\alg}{\textrm{alg}}
\newcommand{\et}{\textrm{\'et}}
\newcommand{\red}{\textrm{red}}

\newcommand{\sm}{\textrm{sm}}

\newcommand{\GL}{\mathrm{GL}}
\newcommand{\SL}{\mathrm{SL}}
\newcommand{\PGL}{\mathrm{PGL}}

\newcommand{\ab}{\operatorname{ab}}
\newcommand{\ad}{\operatorname{ad}}
\newcommand{\alb}{\operatorname{alb}}
\newcommand{\Alb}{\operatorname{Alb}}
\newcommand{\biralb}{\mathfrak{alb}}
\newcommand{\birAlb}{\mathfrak{Alb}}
\newcommand{\Amp}{\operatorname{Amp}}

\newcommand{\aut}{\mathbf{Aut}}
\newcommand{\Aut}{\operatorname{Aut}}
\newcommand{\Bigcone}{\operatorname{Big}}

\newcommand{\CH}{\operatorname{CH}}

\newcommand{\codim}{\operatorname{codim}}

\newcommand{\dr}{\operatorname{dr}}
\newcommand{\Eff}{\overline{\operatorname{Eff}}}
\newcommand{\End}{\operatorname{End}}

\newcommand{\id}{\operatorname{id}}
\newcommand{\im}{\operatorname{Im}}
\newcommand{\isom}{\simeq}
\newcommand{\Ker}{\operatorname{Ker}}
\newcommand{\Mat}{\operatorname{M}}

\newcommand{\Nef}{\operatorname{Nef}}
\newcommand{\NS}{\operatorname{NS}}
\newcommand{\num}{\equiv}

\newcommand{\pic}{\mathbf{Pic}}
\newcommand{\Pic}{\operatorname{Pic}}

\newcommand{\PL}{\operatorname{PL}}

\newcommand{\proj}{\operatorname{proj}}

\newcommand{\rank}{\operatorname{rank}}
\newcommand{\Sing}{\operatorname{Sing}}
\newcommand{\Stab}{\operatorname{Stab}}

\newcommand{\Sym}{\operatorname{Sym}}

\newcommand{\wnum}{\equiv_w}

\newcommand{\A}{A}
\newcommand{\N}{N}
\newcommand{\W}{W}
\newcommand{\Z}{Z}

\begin{document}

\title[A theorem of Tits type for automorphism groups in arbitrary characteristic]{A theorem of Tits type for automorphism groups of projective varieties in arbitrary characteristic  \\[1em] \normalfont\footnotesize With an appendix by Tomohide Terasoma}

\author{Fei Hu}

\address{Department of Mathematics, University of British Columbia, 1984 Mathematics Road, Vancouver, BC V6T 1Z2, Canada
\endgraf Pacific Institute for the Mathematical Sciences, 2207 Main Mall, Vancouver, BC V6T 1Z4, Canada}
\email{\href{mailto:fhu@math.ubc.ca}{\tt fhu@math.ubc.ca}}

\address{Graduate School of Mathematical Sciences, The University of Tokyo, 3-8-1 Komaba, Meguro, Tokyo 153-8914, Japan}
\email{\href{mailto:terasoma@ms.u-tokyo.ac.jp}{\tt terasoma@ms.u-tokyo.ac.jp}}

\begin{abstract}
We prove a theorem of Tits type for automorphism groups of projective varieties over an algebraically closed field of arbitrary characteristic, which was first conjectured by Keum, Oguiso and Zhang for complex projective varieties.
\end{abstract}




\thanks{The author was partially supported by a UBC-PIMS Postdoctoral Fellowship.}

\maketitle



\section{Introduction} \label{section-intro}


\noindent
In 1972, Jacques Tits \cite{Tits72} proved his famous alternative theorem for linear groups. Namely, a finitely generated linear group either contains a non-abelian free subgroup or has a solvable subgroup of finite index.
Later, Keum, Oguiso and Zhang \cite{KOZ09} raised a conjecture of Tits type for automorphism groups of compact K\"ahler manifolds or complex projective varieties with mild singularities.
The first case was soon proved by Zhang \cite{Zhang09} and further generalized in \cite{CWZ14} showing that the automorphism group of a compact K\"ahler manifold satisfies the Tits alternative theorem.
See Dinh \cite{Dinh12} for a survey.

In this paper, we work over an algebraically closed field $\bk$ of arbitrary characteristic and prove a theorem of Tits type for automorphism groups of projective varieties.
The study on the dynamics of projective varieties in positive characteristic has attracted a lot of attention in recent years. For instance, see \cite{ES13, EOY16, Xie15} for the dynamics of surfaces and also \cite{Truong1605, Dang17} for the higher-dimensional case.

Throughout, an algebraic variety is an irreducible reduced separated scheme of finite type over $\bk$ as in \cite{GTM52}.
It is well known that the automorphism group scheme $\aut_X$ of a projective variety $X$ is locally of finite type over $\bk$ and $\Aut(X) = \aut_X(\bk)$; in particular, the reduced neutral component $(\aut^0_X)_{\red}$ of $\aut_X$ is a smooth algebraic group over $\bk$ (see e.g. \cite[\S7]{Brion17}). Denote $(\aut^0_X)_{\red}(\bk)$ by $\Aut^0(X)$.

\begin{theorem} \label{theorem-Tits}
Let $X$ be a projective variety of dimension $n\ge 2$, and $G$ a subgroup of $\Aut(X)$.
Suppose that $G$ does not contain any non-abelian free subgroup.
Then the following assertions hold.
\begin{enumerate}[{\em (1)}]
\item \label{theorem-Tits-ass1} There is a finite-index subgroup $G_1$ of $G$ such that the induced group $G_1|_{\NS_\bR(X)}$ is solvable and Z-connected.
Moreover, let $N(G_1)$ denote the subset of $G_1$ consisting of automorphisms in $G_1$ of null entropy.
Then $N(G_1)$ is a normal subgroup of $G_1$ such that the quotient group $G_1/N(G_1)$ is a free abelian group of rank $r \le n - 1$.
\item \label{theorem-Tits-ass2} Suppose further that $G^0 \coloneqq G \cap \Aut^0(X)$ is finitely generated.
Then $G$ is virtually solvable (i.e., it has a solvable subgroup of finite index).
\end{enumerate}
\end{theorem}

Here $\NS(X)$ denotes the N\'eron--Severi group of $X$, which is a finitely generated abelian group (cf. \cite[Expos\'e~XIII, Th\'eor\`eme~5.1]{SGA6}).
For a field $\bF = \bQ$, $\bR$ or $\bC$,
the $\bF$-vector space $\NS_\bF(X)$ stands for $\NS(X) \otimes_\bZ \bF$.
Inspired by \cite[\S3]{Dinh12} and \cite[\S6.3]{ES13}, we define the {\it first dynamical degree} of an automorphism $g\in\Aut(X)$ as the {\it spectral radius} of its natural pullback action $g^*$ on $\NS_\bR(X)$, i.e.,
\begin{equation*}
d_1(g) \coloneqq \rho\Big(g^*|_{\NS_\bR(X)} \Big) \coloneqq \max\Big\{|\lambda| : \lambda\textrm{ is an eigenvalue of } g^*|_{\NS_\bR(X)} \Big\}.
\end{equation*}
We say that $g$ is of {\it positive entropy} if $d_1(g) > 1$, otherwise it is of {\it null entropy}.
See \S\ref{subsec-dyn-deg} for more details.

The induced group $G|_{\NS_\bR(X)}$ (or $G|_{\NS_\bC(X)}$) is called {\it Z-connected} if its Zariski closure in $\GL(\NS_\bC(X))$ is connected with respect to the Zariski topology.
Note that being Z-connected is only a technical condition for us to apply the theorem of Lie--Kolchin type for a cone in \cite{KOZ09}.
Actually, it is always satisfied by replacing the group with a finite-index subgroup (see e.g. \cite[Remark~3.10]{DHZ15}).

\begin{remark}
(1) In the assertion (\ref{theorem-Tits-ass1}), the rank of the quotient group $G_1/N(G_1)$ is called the {\it dynamical rank of $G_1$} and denoted by $\dr(G_1)$.
It turns out that $\dr(G_1)$ does not depend on the choice of $G_1$.
Hence by a slight abuse of notation, it is also called the {\it dynamical rank of $G$} and denoted by $\dr(G)$.
Sometimes, we may write $\dr(G|_X)$ to emphasize that it is the dynamical rank of the group $G$ acting on $X$.
Generally, given a $G$-action on some algebraic variety $V$, we denote by $G|_V$ the image of $G$ in $\Aut(V)$.

(2) Over $\bC$, the upper bound $\dr(G) \le n-1$ is optimal as seen in \cite[Example~4.5]{DS04}.
Their construction could be generalized to arbitrary characteristic; see our Example \ref{eg:maximal-rank}.
We also refer the interested reader to \cite{Hu-ecdg} for a more general consideration.

(3) The assertion (\ref{theorem-Tits-ass2}) may be false without the finite generation condition.
For instance, let $X$ be the projective space $\bP_\bk^n$ of dimension $n$ over $\bk$. Then $\Aut(X) = \PGL_{n+1}(\bk)$ is a linear algebraic group.
However, the Tits alternative theorem for linear groups in positive characteristic does require a finite generation condition (see \cite[Corollary~1]{Tits72}; indeed, $\GL_n(\wbar \bF_p)$ is a counterexample noticed by Tits \cite{Tits72}).
Alternatively, one may expect a counterpart of the Tits alternative theorem for any finitely generated subgroup $G$ of $\Aut(X)$.
The main obstruction along this direction is that, in general, a subgroup of a finitely generated group $G$ may not be finitely generated.
Indeed, see Appendix~\ref{appendix} for a concrete example due to Terasoma (note that in his notation, $\Aut^0(X)$ is slightly different with ours; but for his explicit example, these two turn out to be the same).
His example gives us a finitely generated solvable subgroup $G$ of $\Aut(X)$ for some projective threefold $X$ such that $G^0 \coloneqq G \cap \Aut^0(X)$ is not finitely generated.
Hence the imposed finite generation assumption on $G^0$ seems to be natural.
\end{remark}

Our next result obtains a better upper bound of the dynamical rank in terms of `Kodaira dimensions'.
See Zhang \cite[Lemma~2.11]{Zhang09} or Dinh \cite[Theorem~1.1]{Dinh12} for related results.
Note however that the proof in \cite{Zhang09} or \cite{Dinh12} hinges on the Deligne--Nakamura--Ueno theorem (cf. \cite[Theorem~14.10]{Ueno75}), which is not known in positive characteristic to the best of our knowledge, not to mention the resolution of singularities.
So simply mimicking Zhang or Dinh's argument does not guarantee the desired upper bound of the dynamical rank.
Thus, one purpose of this paper is to show how the proof given in \cite{Zhang09} or \cite{Dinh12} can be modified so as to circumvent \cite{Ueno75} and Hironaka.
Further, inspired by Dinh--Sibony \cite{DS04}, we are particularly interested in the case when the dynamical rank is maximal, i.e., $\dr(G) = n-1$.
The theorem below also gives some numerical constraints.

\begin{theorem} \label{theorem-max-rank}
Let $X$ be a projective variety of dimension $n\ge 2$,
and $G$ a subgroup of $\Aut(X)$ such that the induced group $G|_{\NS_\bR(X)}$ is solvable and Z-connected.
Let $\nu \colon X^\nu \to X$ be the normalization of $X$.
\begin{enumerate}[{\em (1)}]
\item \label{theorem-max-rank-ass1} Then the dynamical rank of $G$ satisfies the following relation:
$$\dr(G) \le \min \Big\{n-1, \max \big\{0, n-1-\kappa(\omega_{X^\nu}) \big\}, \max \big\{0, n-1-\kappa(X) \big\} \Big\},$$
where $\kappa(\omega_{X^\nu})$ and $\kappa(X)$ denote the Kodaira--Iitaka dimension of $X^\nu$ and the Kodaira dimension of $X$, respectively.
\item \label{theorem-max-rank-ass2} Suppose that $\dr(G) = n-1$.
Then $\kappa(\omega_{X^\nu}) \le 0$, $\kappa(X) \le 0$, the irregularity $q(X) = 0$ or $n$, and the birational irregularity $\tilde q(X) = 0$ or $n$.
\end{enumerate}
\end{theorem}

See \S\ref{subsec-kappa-q} for the precise definitions of $\kappa(\omega_{X^\nu})$, $\kappa(X)$, $q(X)$ and $\tilde q(X)$.
Here we have two kinds of `Kodaira dimension'.
Indeed, there is no `canonical' definition of the Kodaira dimension for general singular varieties since the resolution of singularities in positive characteristic is still unknown (at least for the dimension greater than three).

The paper is organized as follows.
After a brief review of the Kodaira dimension and the (birational) irregularity in \S\ref{subsec-kappa-q}, as well as the intersection theory on singular varieties in \S\ref{subsec-cycle}, we shall introduce a non-standard notion, so-called `weak numerical equivalence', which allows us to generalize \cite{DS04} in our context.
A crucial point is that it is equivalent to the usual numerical equivalence for nef classes (see Lemma~\ref{lemma-num-weak-num}).
Then in \S\ref{subsec-dyn-deg}, we define the $k$-th dynamical degree which turns out to be a birational invariant.
In \S\ref{subsec-HIT}, we prove a higher-dimensional Hodge index theorem for $\bR$-Cartier divisors.
Note that this is the main ingredient utilized by Dinh and Sibony in \cite{DS04}.
In \S\ref{subsec-quasi-nef}, we give a review of the quasi-nef sequence which plays an essential role in Zhang~\cite{Zhang09}.
Finally, we prove Theorems~\ref{theorem-Tits} and \ref{theorem-max-rank} in Sections~\ref{sec-proof-Tits} and \ref{sec-proof-max-rank}, respectively.


\section{Preliminaries} \label{sec-prelim}


\noindent
Throughout this section, unless otherwise stated, $X$ is a projective variety of dimension $n$.
We shall follow Hartshorne \cite{GTM52} for the standard definitions and notation in algebraic geometry.
Besides, we refer to Lazarsfeld \cite{Lazarsfeld-I} for the theory of $\bR$-Cartier divisors and Fulton \cite{Fulton98} for the intersection theory on general singular varieties.

\subsection{Kodaira dimension and irregularity} \label{subsec-kappa-q}

\subsubsection{Kodaira dimension and Kodaira--Iitaka dimension} \label{subsubsec-kappa}

Recall that the Kodaira dimension of a singular variety in characteristic zero is usually defined as the Kodaira dimension of any smooth model (see e.g. \cite[Example~2.1.5]{Lazarsfeld-I}).
However, due to the absence of the resolution of singularities in positive characteristic, there is some slight difficulty to define the Kodaira dimension of those singular varieties in positive characteristic.
Here we list the following two reasonable definitions.

The {\it Kodaira dimension} $\kappa(X)$ of $X$ is defined as in \cite[Definition~5.1]{Luo87}.
Note that Luo's definition depends only on the (rational) function field $\bk(X)$ of $X$.
So the main advantage of this definition is that $\kappa(X)$ is indeed a birational invariant of $X$.
Also, this definition coincides with Abramovich's (ad hoc) definition (i.e., \cite[Definition~1, p.~46]{Abramovich94}), at least for subvarieties of (semi)abelian varieties.
We will make use of this definition in the proof of Lemma~\ref{lemma-max-rank-q}.

Another definition of Kodaira dimension is more or less in the usual sense by using the `canonical sheaf'.
Suppose that $X$ is a {\it normal} projective variety.
Then its {\it canonical sheaf} is defined by $\omega_X \coloneqq j_*\,\omega_{X_{\sm}}$,
where $X_{\sm}$ is the smooth locus of $X$ and $j \colon X_{\sm} \injmap X$ is the corresponding open immersion.
Note that $\omega_X$ is a reflexive sheaf of rank one which gives rise to a Weil divisor, the {\it canonical divisor} $K_X$, such that $\sO_X(K_X) \sim \omega_X$.
The {\it Kodaira--Iitaka dimension} $\kappa(\omega_X)$ of $X$ is then defined as the Iitaka dimension $\kappa(X, \omega_X)$ of the canonical sheaf $\omega_X$ (or equivalently, the Iitaka-$D$ dimension of the canonical divisor $K_X$).
Recall that
\begin{equation*}
\kappa(X, \omega_X) \coloneqq
\left\{
\begin{array}{ll}
-\infty & \text{if $H^0(X, \omega_X^{[m]}) = 0$ for all $m > 0$,} \\
\max\Big\{\dim\big(\phi_{\omega_X^{[m]}}(X)\big) : m > 0 \Big\} & \text{otherwise,}
\end{array}
\right.
\end{equation*}
where $\omega_X^{[m]} \coloneqq (\omega_X^{\otimes m})^{\vee\vee}$ is the double dual of $\omega_X^{\otimes m}$.
We refer to \cite[Appendices~A and B]{Patakfalvi18} for some standard properties of $\kappa(X)$ and $\kappa(\omega_X)$.
For instance, \cite[Proposition~B.1]{Patakfalvi18} asserts that $\kappa(X) \le \kappa(\omega_X)$.
The main advantage of the second definition is the existence of the Iitaka fibration as in characteristic zero (see e.g. \cite[Lemma~A.7]{Patakfalvi18}).

In an analogous way, we can also define the {\it anti-Kodaira--Iitaka dimension} $\kappa^{-1}(\omega_X)$ as the Iitaka dimension of the anti-canonical sheaf $\omega_X^{-1}$, i.e., $\kappa^{-1}(\omega_X) \coloneqq \kappa(X, \omega_X^{-1})$.

\subsubsection{Albanese variety, Albanese morphism and irregularity} \label{subsubsec-alb}

Due to Serre \cite[Th\'eor\`eme~5]{Serre58},
there exist an abelian variety $\Alb(X)$
and a morphism $\alb_X \colon X \to \Alb(X)$
such that any morphism from $X$ to an abelian variety factors,
uniquely up to translations, through this $\Alb(X)$.
Then the abelian variety $\Alb(X)$ (resp. the morphism $\alb_X$) is called the {\it Albanese variety} (resp. the {\it Albanese morphism}) of $X$.
Note, however, that this construction of the Albanese morphism is, in general, not of a birational nature.
Alternatively, one can (birationally) define the {\it birational Albanese variety} $\birAlb(X)$ and the {\it Albanese map} $\biralb_X \colon X \ratmap \birAlb(X)$ using the universal property in an obvious way (cf. \cite[Chapter~II, \S3]{Lang83}).
The {\it irregularity} $q(X)$ and the {\it birational irregularity} $\tilde q(X)$ of $X$ are then defined as follows:
$$q(X) \coloneqq \dim \Alb(X), \ \ \tilde q(X) \coloneqq \dim \birAlb(X).$$
It is easy to see that $\alb_X$ factors through $\biralb_X$ by the universal property so that there is a surjective morphism $\birAlb(X) \to \Alb(X)$ of abelian varieties.
In particular, $\tilde q(X) \ge q(X)$.
See \cite[Th\'eor\`eme~6]{Serre58} for more relations between these two `Albanese varieties'.

If we assume that $X$ is a {\it normal} projective variety, then by \cite[Theorem~9.5.4]{Kleiman05}, the neutral component $\pic^0_{X/\bk}$ of the Picard scheme $\pic_{X/\bk}$ is a projective group scheme over $\bk$ (but it may not be reduced in positive characteristic).
Nevertheless, the reduction $(\pic^0_{X/\bk})_{\red}$ of $\pic^0_{X/\bk}$ is an abelian variety which is the dual abelian variety of the Albanese variety $\Alb(X)$; see \cite[Remark~9.5.25]{Kleiman05}.
Let $T_0 \pic^0_{X/\bk}$ denote the Zariski tangent space of $\pic^0_{X/\bk}$ at $[\sO_X]$,
which is canonically isomorphic to $H^1(X, \sO_X)$ (cf. \cite[Theorem~9.5.11]{Kleiman05}).
Then we have the following numerical relations:
$$q(X) = \dim \Alb(X) = \dim (\pic^0_{X/\bk})_{\red} = \dim \pic^0_{X/\bk} \le \dim T_0 \pic^0_{X/\bk} = h^1(X, \sO_X).$$
The equality holds if and only if $\pic^0_{X/\bk}$ is reduced, i.e., $\pic^0_{X/\bk}$ itself is an abelian variety.

\subsection{Numerical classes and dual classes} \label{subsec-cycle}

An {\it algebraic cycle of dimension $k$} (or a {\it $k$-cycle}) on $X$ is an element of the free abelian group $\Z_k(X)$ generated by closed subvarieties of $X$ of dimension $k$.
Instead of studying $\Z_k(X)$ directly, we work on the quotient groups of $\Z_k(X)$ by some equivalence relations.
The first one is rational equivalence (denoted by $\sim$;
see \cite[\S1.3]{Fulton98} for its precise definition).
The group of $k$-cycles modulo rational equivalence is called the {\it Chow group $\A_k(X)$}.
Its element is called a {\it $k$-cycle class}.
In \cite[Chapter~3]{Fulton98}, Fulton constructed the following Chern classes operations
$$c_i(\sE) \cap - \colon A_k(X) \to A_{k-i}(X),$$
for any vector bundle $\sE$ on $X$.
It thus induces to define the numerical equivalence on algebraic cycles as follows.

\begin{definition}[{cf. \cite[Definition~19.1]{Fulton98}}] \label{def-num-equiv}
An algebraic cycle $Z \in \Z_k(X)$ is called {\it numerically trivial} and denoted by $Z \num 0$, if
\begin{equation} \label{eqn-num}
P(\sE_I) \cap Z = 0
\end{equation}
for all homogeneous polynomials $P(\sE_I)$ of degree $k$ in Chern classes of vector bundles $\sE_i$ on $X$.
\end{definition}

Let $\N_k(X)_\bZ$ denote the quotient group of $\Z_k(X)$ modulo numerical equivalence, i.e.,
\[\N_k(X)_\bZ \coloneqq \Z_k(X) / \! \num \! .\]
It is known that $\N_k(X)_\bZ$ is a free abelian group of finite rank (cf. \cite[Example~19.1.4]{Fulton98}).
In this paper, we will consider the following finite-dimensional real vector space
\[\N_k(X) \coloneqq \N_k(X)_\bZ \otimes_\bZ \bR.\]
A {\it numerical class} is an element of $N_k(X)$.
Denote the numerical class of a real $k$-cycle $Z$ by $[Z]$.

Since our varieties may be singular, Chow's moving lemma is invalid so that the usual non-degenerate pairing $N_k(X) \times N_{n-k}(X) \to \bR$ does not exist.
Instead, it turns out to be quite useful to consider the abstract dual groups $Z^k(X)$ and $N^k(X)$ of $Z_k(X)$ and $N_k(X)$, respectively.
Also, $N^k(X)$ is formally identified with the group of homogeneous $\bR$-polynomials of degree $k$ in Chern classes of vector bundles on $X$ modulo numerical equivalence \eqref{eqn-num}.
A {\it numerical dual class} is an element of $N^k(X)$.
In particular, $N^1(X)$ is isomorphic to the {\it N\'eron--Severi space} $\NS_\bR(X)$ of $\bR$-Cartier divisors modulo numerical equivalence (see e.g. \cite[Example~2.1]{FL17}).
Usually, $[D]$ represents the numerical class of an $\bR$-Cartier divisor $D$ in $\N^1(X)$.
But by abuse of notation, we shall write $D \in \N^1(X)$.

By sending $[P(\sE_I)]$ to $P(\sE_I) \cap [X]$, we have the following natural map
\begin{equation*} \label{eqn-phi}
\phi \colon N^{k}(X) \to N_{n-k}(X).
\end{equation*}
When $k = 1$, this is just the numerical version of the cycle map from Cartier divisors to Weil divisors,
and it is injective as we shall see in Proposition~\ref{prop-HIT-R}.
In general, $\phi$ may not be an isomorphism for singular varieties (see e.g. \cite[Example~2.8]{FL17}).

\begin{convention}
When there is no confusion, we will sometimes use $\cdot$ instead of $\cap$ to denote the intersection of Chern classes (or Cartier divisors) with cycles and omit the cap product with $[X]$.
\end{convention}

\subsubsection{Numerical pushforward and pullback} \label{subsubsec-num-push-pull}

Let $\pi \colon X \to Y$ be a proper morphism of projective varieties.
Then we have a well-defined (numerical) pushforward homomorphism $\pi_* \colon N_k(X) \to N_k(Y)$
which is naturally induced from the proper pushforward $\pi_* \colon Z_k(X) \to Z_k(Y)$ of cycles by the following projection formula (cf. \cite[Theorem~3.2(c)]{Fulton98}):
\begin{equation} \label{eqn-proj-formula}
\pi_*(P(\pi^*\sE_I)\cap Z) = P(\sE_I) \cap \pi_*Z.
\end{equation}
It thus induces the (numerical) proper pullback $\pi^* \colon N^k(Y) \to N^k(X)$ which is dually defined by $(\pi_*)^\vee$.
When $\pi$ is surjective, it is known that $\pi_*$ is surjective (see e.g. \cite[Remark~2.13]{FL17}).
Moreover, \cite[Corollary~3.22]{FL17} asserts the surjectivity of $\pi_* \colon \Eff_k(X) \to \Eff_k(Y)$ for all $k$ (see \S\ref{subsubsec-pos-cone} for the meaning of the notation $\Eff_k$).
Hence dually, we also have the injectivity of $\pi^*$.
We refer to \cite[Chapter~19]{Fulton98} and \cite[\S 2.1]{FL17} for more details about the intersection theory on singular varieties.

\subsubsection{Positive cones} \label{subsubsec-pos-cone}

A $k$-cycle $Z \in Z_k(X)$ is {\it effective}, if all of its defining coefficients are non-negative.
The corresponding numerical class $[Z] \in N_k(X)$ is called an {\it effective numerical class}.
We denote by $\Eff_k(X)$ the closure of the convex cone generated by all effective numerical classes in $N_k(X)$.
It is called the {\it pseudo-effective cone} of $N_k(X)$.
The cone dual to $\Eff_k(X)$ in $N^k(X)$ is called the {\it nef cone} $\Nef^k(X)$,
which is a salient closed convex cone of {\it full dimension} (i.e., it generates $N^k(X)$ as a vector space).
An element of $\Nef^k(X)$ is called a {\it nef class}.
In particular, $\Nef^1(X)$ is the usual nef cone $\Nef(X)$ consisting of all nef $\bR$-Cartier divisor classes.
A numerical dual class $\gamma \in N^k(X)$ is {\it pseudo-effective} if $\phi(\gamma)$ is pseudo-effective in $N_{n-k}(X)$, i.e., $\gamma \cap [X] \in \Eff_{n-k}(X)$.
Denote by $\Eff^k(X)$ the closed cone of all pseudo-effective dual classes in $N^k(X)$.
Conventionally, the {\it ample cone $\Amp^k(X)$} (resp. the {\it big cone $\Bigcone_k(X)$}) is the interior of the nef cone $\Nef^k(X)$ (resp. the pseudo-effective cone $\Eff_k(X)$).

Besides, there is another salient closed convex cone in $N^k(X)$ of full dimension, the {\it pliant cone} $\PL^k(X)$;
see \cite[Definition~3.2]{FL17} for its precise definition.
When $k = 1$, it coincides with the nef cone $\Nef^1(X)$.
When $k \ge 2$, there exist nef classes which are not even pseudo-effective.
However, this pliant cone behaves much better than the nef cone.
We have the following relation (cf. \cite[Lemma~3.7]{FL17}):
$$\PL^k(X) \subset \Eff^k(X) \cap \Nef^k(X).$$
It should be mentioned that if $H_1, \ldots, H_k \in N^1(X)$ are ample classes, then $H_1 \cdots H_k \in N^k(X)$ is in the ample cone $\Amp^k(X)$ (actually, it is in the interior of the pliant cone $\PL^k(X)$; see \cite[Lemma~3.14]{FL17}), and $H_1 \cdots H_k \cap [X] \in N_{n-k}(X)$ is big (cf. \cite[Lemma~2.12]{FL17}).
We refer to Fulger--Lehmann \cite{FL17} for more properties about those positive cones.

\subsection{Weak numerical equivalences} \label{subsec-weak-num}

In this paper, we also need the following weak numerical equivalence on algebraic cycles (implicitly) introduced by Zhang \cite{Zhang09}.

\begin{definition} \label{def-weak-num}
An algebraic cycle $Z \in \Z_k(X)$ is called {\it weakly numerically trivial} and denoted by $Z \wnum 0$, if
$$P(\sL_I) \cap Z = 0$$
for all homogeneous polynomials $P(\sL_I)$ of degree $k$ in the first Chern classes of line bundles $\sL_i$ on $X$
(equivalently, it can be interpreted using the language of Cartier divisors as following:
\begin{equation*}
Z \wnum 0, \, \text{if } H_1 \cdots H_k \cdot Z = 0
\end{equation*}
for all ample divisors and hence for all Cartier divisors $H_i$ on $X$).
\end{definition}

\begin{remark} \label{remark-weak-num}
It is straightforward to see that weak numerical equivalence is indeed weaker than numerical equivalence by definitions.
On the other hand, we can also define the dual version of weak numerical equivalence on $Z^k(X)$.
More precisely, a dual cycle $\gamma \in \Z^k(X)$ is {\it weakly numerically trivial},
if $\gamma \cap [X]$ is weakly numerically trivial in the sense of Definition~\ref{def-weak-num}.
That is, $H_1 \cdots H_{n-k} \cdot (\gamma \cap [X]) = 0$ for all ample divisors $H_i$ on $X$.
Equivalently, by the commutativity of the cap product (cf. \cite[Theorem~3.2(b)]{Fulton98}),
$\gamma \cdot (H_1 \cdots H_{n-k} \cap [X]) = 0$ for all ample $\bR$-divisors $H_i$ on $X$.
Thus, on $Z^k(X)$, weak numerical equivalence is also weaker than numerical equivalence.
In particular, the product $D_1 \cdots D_k \in N^k(X)$ of $\bR$-Cartier divisors $D_1, \dots, D_k$ is weakly numerically trivial, if $D_1 \cdots D_k \cdot H_1 \cdots H_{n-k} = 0$
for all ample $\bR$-divisors $H_i$ on $X$.

Let $\W^k(X)$ denote the quotient group of $\Z^k(X)$ modulo weak numerical equivalence, i.e.,
\[\W^k(X) \coloneqq \Z^k(X) / \! \wnum \! .\]
Then $\W^k(X) = \N^k(X) / \! \wnum$ is a finite-dimensional real vector space since so is $N^k(X)$.
\end{remark}

Surprisingly, it turns out that for nef classes weak numerical equivalence coincides with numerical equivalence.
To see this, we need the following Bertini-type result
which asserts the existence of general hypersurface sections containing a given closed subvariety.
See \cite[Theorem~1]{KA79} or \cite[Theorem~8.1]{DR14} for a similar treatment.
For the convenience of the reader, here we give its proof.

\begin{lemma} \label{lemma-Bertini}
Let $Z$ be a closed subvariety of $X$ of codimension $\ge 1$.
Then for any ample divisor $H$ on $X$, there exist an integer $d\gg 0$ and an ample divisor $H_d$ in $|dH|$ such that
$H_d$ contains $Z$.
If we assume further that $X$ is normal and $\codim(Z, X)\ge 2$, then the above $H_d$ is a projective subvariety of $X$.
\end{lemma}

\begin{proof}
After replacing $H$ by some large multiple, we may assume that $H$ is very ample and hence the complete linear system $|H|$ defines a closed immersion $\Phi_{|H|} \colon X \injmap \bP W$, where $W \coloneqq H^0(X, \sO_X(H))$.
Let $\sI_Z$ denote the ideal sheaf of $Z$ in $X$ which is a coherent sheaf.
Then we have an exact sequence of $\sO_X$-modules
$0 \to \sI_Z \to \sO_X \to \sO_Z \to 0$.
After tensoring with $\sO_X(dH)$, we have
$$0 \to \sI_Z \otimes \sO_X(dH) \to \sO_X(dH) \to \sO_Z(dH) \to 0.$$
By Serre's vanishing theorem \cite[Chapter~III, Proposition~5.3]{GTM52}, $H^i(X, \sI_Z \otimes \sO_X(dH)) = 0$ for $d\gg 0$ and $i>0$.
Hence we have the following exact sequence of finite-dimentional $\bk$-vector spaces
$$0 \to H^0(X, \sI_Z \otimes \sO_X(dH)) \to H^0(X, \sO_X(dH)) \to H^0(Z, \sO_Z(dH)) \to 0.$$
Let $V_d$ denote $H^0(X, \sI_Z \otimes \sO_X(dH))$.
Then it follows from the asymptotic Riemann--Roch formula that $\dim_{\bk} V_d \sim d^{\,\dim X}$ for $d\gg 0$ (cf. \cite[Example~1.2.19]{Lazarsfeld-I}).
So the linear system $\mathfrak{d}_d \subset |dH|$ on $X$ corresponding to the subspace $V_d$ is non-empty.
Also, note that $\sI_Z \otimes \sO_X(dH)$ is generated by global sections for $d\gg 0$ by \cite[Chapter~II, Theorem~5.17]{GTM52} (see also \cite[Theorem~1.2.6]{Lazarsfeld-I}).
Then the base locus of $\mathfrak{d}_d$ is exactly $Z$ for $d\gg 0$.
Thus $\mathfrak{d}_d$ defines a nontrivial morphism
$\iota_d \colon X\setminus Z \longrightarrow \bP V_d.$

We claim that $\iota_d$ is a locally closed immersion\footnote{A morphism is called a {\it locally closed immersion} (or an {\it immersion} for short), if it can be factored as a closed immersion followed by an open immersion.
To prove the claim, two useful facts will be implicitly used.
One is that a composition of immersions of schemes is an immersion (cf. \cite[\href{http://stacks.math.columbia.edu/tag/02V0}{Tag~02V0}]{stacks-project}).
Another one is as follows.
Let $X_1 \to X_2 \to X_3$ be morphisms of schemes.
If $X_1 \to X_3$ is an immersion, then $X_1 \to X_2$ is an immersion (see e.g. \cite[\href{http://stacks.math.columbia.edu/tag/07RK}{Tag~07RK}]{stacks-project}).} for $d\gg 0$.
In fact, since $X\setminus Z \injmap \bP W$ is locally closed, so is $X\setminus Z \to \bP V_d \times \bP W$.
Let $\bP V_d \times \bP W \injmap \bP(V_d \otimes W)$ be the Segre embedding.
Then the composite morphism $X\setminus Z \to \bP(V_d \otimes W)$ is locally closed.
Note that the natural map $V_d \otimes W \to V_{d+1}$ is surjective for $d\gg 0$ (see e.g. \cite[Theorem~1.8.5]{Lazarsfeld-I}).
So $X\setminus Z \to \bP(V_d \otimes W)$ factors through the closed immersion $\bP V_{d+1} \injmap \bP(V_d \otimes W)$.
Hence $\iota_{d+1}$ is locally closed and the claim follows.
Let $\Sing X$ denote the singular locus of $X$.
Then by applying Bertini's theorem \cite[Corollaire~6.11]{Jouanolou83} to the composite morphism $X\setminus (Z \cup \Sing X) \injmap \bP V_d$ which is still locally closed and hence unramified,
a general member $H_d \in \mathfrak{d}_d$ is smooth and irreducible outside $Z \cup \Sing X$.

By the assumption that $X$ is normal, $\codim(\Sing X, X)\ge 2$.
So if $\codim(Z, X)\ge 2$, then the above $H_d$ itself is irreducible and has the property $(R_0)$.
To see $H_d$ is reduced, it suffices to show that it also satisfies the Serre's condition $(S_1)$.
This follows from the $(S_2)$-ness of $X$ by Serre's criterion for normality and a general fact that a (locally) hypersurface section of an $(S_d)$-scheme is $(S_{d-1})$.
Therefore, our $H_d$ is integral and hence a projective subvariety of $X$.
We have proved Lemma~\ref{lemma-Bertini}.
\end{proof}

\begin{rmk}
The above $H_d$ may not be chosen so that it is still normal as one expects. Actually, in our proof, neither the condition $(R_1)$ nor Serre's condition $(S_2)$ is necessarily satisfied.
A sufficient condition to ensure the normality of $H_d$ could be that $\codim(Z, X) \ge 3$ and $X$ has properties $(R_2)$ and $(S_3)$.
\end{rmk}

\begin{lemma} \label{lemma-num-weak-num}
Let $\gamma \in \Nef^k(X)$ be a nef class in $N^k(X)$ such that $\gamma \not\num 0$.
Then $\gamma \cdot H_1 \cdots H_{n-k} > 0$ for all ample $\bR$-divisors $H_i$ on $X$.
In particular, $\gamma \not\wnum 0$.
\end{lemma}

\begin{proof}
It suffices to consider the case that each $H_i$ is an ample Cartier divisor.
We shall prove this lemma by the induction on $\dim X$.
When $\dim X = 2$, it is true by the Hodge index theorem for surfaces.
Suppose that the lemma has been proved for any projective variety of dimension less than $n$.
Replacing $X$ by its normalization, we may assume that $X$ is normal.
By the definition of numerical equivalence \eqref{eqn-num}, there exists an integral $k$-cycle $Z$ on $X$ such that $\gamma \cdot Z \ne 0$ and hence $\gamma \cdot Z > 0$ as $\gamma\in \Nef^k(X)$.
If $k = n-1$, by Lemma~\ref{lemma-Bertini}, there exist an integer $d\gg 0$ and an ample divisor $D \in |dH_1|$ containing $Z$.
It follows that $\gamma \cdot dH_1 = \gamma \cdot D = \gamma \cdot ([Z] + \alpha') \ge \gamma \cdot [Z] > 0$, where $\alpha'$ is an effective class.
If $k \le n-2$, again by Lemma~\ref{lemma-Bertini}, we can choose $D \in |dH_1|$ for $d\gg 0$ as a projective subvariety of $X$ containing $Z$.
Let $i \colon D \injmap X$ be the closed immersion.
Then $i^*\gamma \in \Nef^k(D)$ and $i^*\gamma \cdot Z = \gamma \cdot i_*Z = \gamma \cdot Z > 0$ so that $i^*\gamma \not\num 0$.
Let $H'_s \coloneqq i^*H_s$ denote the restriction of the ample divisor $H_s$ to $D$ with $2\le s\le n-k$.
Then by the hypothesis induction, we have $i^*\gamma \cdot H'_2 \cdots H'_{n-k} > 0$.
On the other hand,
\begin{align*}
&\ \ i^*\gamma \cdot H'_2 \cdots H'_{n-k} = i^*\gamma \cdot i^*H_2 \cdots i^*H_{n-k} \cap [D] \\
&= \gamma \cdot H_2 \cdots H_{n-k} \cap i_*[D] = \gamma \cdot H_2 \cdots H_{n-k} \cap (D\cap [X]) \\
&= \gamma \cdot dH_1 \cdot H_2 \cdots H_{n-k} \cap [X] = d\gamma \cdot H_1 \cdot H_2 \cdots H_{n-k},
\end{align*}
where $i_*[D] = i_*i^*[X] = D \cap [X]$ (cf. \cite[Proposition~2.6]{Fulton98}). Thus the lemma follows.
\end{proof}

\subsection{Dynamical degrees} \label{subsec-dyn-deg}

We first recall the notion of dynamical degrees in complex dynamical systems. Given a compact K\"ahler manifold $M$ of dimension $n$, for any integer $0\le p\le n$, the $p$-th dynamical degree $d_p(f)$ of a holomorphic automorphism $f$ of $M$ is the spectral radius of the linear transformation $f^* \colon H^{p,p}(M, \bC) \to H^{p,p}(M, \bC)$.
It is known that the topological entropy $h_{\rm top}(f)$ of $f$ is then equal to $\max_{0\le p\le n} \log d_p(f)$ by the Gromov--Yomdin theorem (cf.~\cite[Th\'eor\`eme~2.1]{DS04}).
See also \cite[\S3]{Dinh12} for a survey on dynamical degrees.
Further, when $M$ is projective, it turns out that the first dynamical degree $d_1(f)$ is equal to the spectral radius of the natural pullback action $f^*|_{\NS_\bR(M)}$.

Back to our arbitrary characteristic setting, it was Esnault and Srinivas who first introduced two natural algebraic definitions of entropy in \cite{ES13} inspired by the above equality $d_1(f) = \rho(f^*|_{\NS_\bR(M)})$.
More precisely, they defined the {\it entropy} $h(g)$ (resp. the {\it algebraic entropy} $h_{\alg}(g)$) of an automorphism $g$ of $X$ to be the natural logarithm of the spectral radius of $g$ acting on the $\ell$-adic cohomology $H_{\et}^{\bullet}(X, \bQ_\ell)$ (resp. on the numerical Chow ring $\CH_{\rm num}^{\bullet}(X)$).
The equivalence of these two notions is still unknown except for smooth projective surfaces (cf.~\cite[\S6.3]{ES13}).
In this paper, we define the {\it $k$-th dynamical degree} by the natural pullback action $g^*$ on $N^k(X)$ for any integer $0\le k\le n$.
Namely,
\begin{equation*}
d_k(g) \coloneqq \rho\Big(g^*|_{\N^k(X)} \Big) = \max\Big\{|\lambda| : \lambda\textrm{ is an eigenvalue of } g^*|_{\N^k(X)} \Big\}.
\end{equation*}
Then the algebraic entropy $h_{\alg}(g) = \log \max\{d_k(g) : 0\le k\le n\}$.
It follows from Corollary~\ref{cor-log-concave} that $h_{\alg}(g) > 0$ if and only if $d_k(g) > 1$ for some (or equivalently, for all) $1\le k\le n-1$.

We show that $d_k(g)$ has an intersection-theoretic characterization as follows (see \cite[\S3]{Dinh12} for an analogue in complex dynamics).

\begin{lemma} \label{lemma-equiv-dyn-deg}
For any $0 \le k \le n$ and any ample $\bR$-divisor $H$ on $X$, we have
\begin{equation*}
d_k(g) = \lambda_k(g) \coloneqq \lim_{m\to \infty} \Big((g^m)^*H^k \cdot H^{n-k}\Big)^{1/m}.
\end{equation*}
In particular, the limit exists and does not depend on the choice of $H$.
\end{lemma}

\begin{proof}
We first define a norm on $N^k(X)$.
Indeed, by \cite[Lemma~2.12]{FL17}, we see that $H^{n-k} \cap [X]$ is big, i.e., it belongs to the interior of the pseudo-effective cone $\Eff_k(X)$.
We can choose pseudo-effective classes $\{\alpha_i\}$ that span $N_k(X)$ and such that $H^{n-k} \cap [X] \num \sum_i \alpha_i$ (see \cite[Proof of Corollary~3.16]{FL17}).
Then for any $\gamma \in N^k(X)$, the assignment $\gamma \mapsto \sum_i \abs{\gamma \cdot \alpha_i}$ is a norm on $N^k(X)$.
In particular, if $\gamma$ is a nef dual class, then $\norm{\gamma} = \sum_i \gamma \cdot \alpha_i = \gamma \cdot H^{n-k} \cap [X]$.
We endow $g^*|_{N^k(X)}$ the induced matrix norm.
Then by the spectral radius formula,
$$d_k(g) = \lim_{m\to \infty} \norm{(g^m)^*|_{N^k(X)}}^{1/m} \ge \limsup_{m\to \infty} \norm{(g^m)^*H^k}^{1/m} = \limsup_{m\to \infty} \Big((g^m)^*H^k \cdot H^{n-k}\Big)^{1/m}.$$

On the other hand, note that the pliant cone $\PL^k(X)$ is a salient closed convex cone in $N^k(X)$ of full dimension and preserved by $g^*|_{\N^k(X)}$ (cf. \cite[\S3]{FL17}).
So by applying the generalized Perron--Frobenious theorem to the triplet $(\PL^k(X), N^k(X), g^*|_{\N^k(X)})$,
there exists an eigenvector $\upsilon_g \in \PL^k(X)$ such that $g^*\upsilon_g \num d_k(g)\upsilon_g$ in $N^k(X)$ (see Theorem~\ref{theorem-Birkhoff-cone}).
We also notice that $H^k$ is in the interior of $\PL^k(X)$ (cf. \cite[Lemma~3.14]{FL17}).
Then so is $c_0H^k - \upsilon_g$ for some $c_0 \gg 0$.
This yields that $(g^m)^*(c_0H^k - \upsilon_g) \in \PL^k(X) \subset \Eff^k(X)$ for any $m$ and hence $(g^m)^*(c_0H^k - \upsilon_g)\cdot H^{n-k}\ge 0$.
We thus have
\begin{align*}
&\ \ \liminf_{m\to \infty} \Big((g^m)^*H^k \cdot H^{n-k}\Big)^{1/m} = \liminf_{m\to \infty} \Big((g^m)^*(c_0H^k) \cdot H^{n-k}\Big)^{1/m} \\
&\ge \liminf_{m\to \infty} \Big((g^m)^*\upsilon_g \cdot H^{n-k}\Big)^{1/m} = \lim_{m\to \infty} \Big(d_k^m(g)\upsilon_g \cdot H^{n-k}\Big)^{1/m} = d_k(g).
\end{align*}
It follows that $\lambda_k(g)$ is a well-defined limit which is equal to $d_k(g)$, and hence it is independent of the choice of $H$.
Thus we prove the lemma.
\end{proof}

Furthermore, we may replace the ample $\bR$-divisor $H$ in Lemma~\ref{lemma-equiv-dyn-deg} by a nef and big divisor.

\begin{lemma} \label{lemma-alt-def-deg}
For any $0 \le k \le n$ and any nef and big $\bR$-Cartier divisor $D$ on $X$, we have
$$d_k(g) = \lim_{m\to \infty} \Big((g^m)^*D^k \cdot D^{n-k}\Big)^{1/m}.$$
\end{lemma}

\begin{proof}
We first fix an ample divisor $A$ on $X$.
Then $c_0A - D$ is still ample for $c_0\gg 0$.
This yields that $(g^m)^*(c_0A)^k \cdot (c_0A)^{n-k} \ge (g^m)^*D^k \cdot D^{n-k}$ for any $m$.
Hence by Lemma~\ref{lemma-equiv-dyn-deg},
$$d_k(g) = \lim_{m\to \infty} \Big((g^m)^*(c_0A)^k \cdot (c_0A)^{n-k}\Big)^{1/m} \ge \limsup_{m\to \infty} \Big((g^m)^*D^k \cdot D^{n-k}\Big)^{1/m}.$$
For the other direction, let $D \num H + E$ be a numerical decomposition of the big $\bR$-Cartier divisor $D$, where $H$ is an ample $\bR$-divisor and $E$ is an effective $\bR$-Cartier divisor on $X$.
The existence of this decomposition is due to Kodaira's lemma (cf. \cite[Proposition~2.2.6]{Lazarsfeld-I}), which is still valid in positive characteristic since its proof depends only on the asymptotic Riemann--Roch formula \cite[Example~1.2.19]{Lazarsfeld-I} and Serre's vanishing theorem \cite[Chapter~III, Proposition~5.3]{GTM52}.
Then by the nefness of $D$, we have
$$(g^m)^*D^k \cdot D^{n-k} = (g^m)^*(H + E)^k \cdot D^{n-k} \ge (g^m)^*H^k \cdot D^{n-k} \ge (g^m)^*H^k \cdot H^{n-k}.$$
This yields that $$\displaystyle \liminf_{m\to \infty} \Big((g^m)^*D^k \cdot D^{n-k}\Big)^{1/m} \ge \liminf_{m\to \infty} \Big((g^m)^*H^k \cdot H^{n-k}\Big)^{1/m} = d_k(g)$$ by Lemma~\ref{lemma-equiv-dyn-deg}.
Hence the lemma follows.
\end{proof}

The lemma below shows that all dynamical degrees of an automorphism are equivalent to the same ones on its normalization or an equivariant resolution (if it exists) and hence are birational invariants.

Let $g$ be an automorphism of $X$.
A morphism $\pi: X \to Y$ is called {\it $g$-equivariant} if $\pi \circ g|_X = g|_Y \circ \pi$, i.e., the $g$-action on $X$ descends to a biregular (possibly non-faithful) action on $Y$ (denoted by $g|_Y$).

\begin{lemma} \label{lemma-inv-dyn-deg}
Let $\pi\colon X\to Y$ be a $g$-equivariant surjective morphism of projective varieties.
Then we have $d_k(g|_X) \ge d_k(g|_Y)$.
Suppose further that $\pi$ is generically finite. Then $d_k(g|_X) = d_k(g|_Y)$.
In particular, $g|_X$ is of positive entropy (resp. null entropy) if and only if so is $g|_Y$.
\end{lemma}

\begin{proof}
By the generalized Perron--Frobenious theorem (see Theorem~\ref{theorem-Birkhoff-cone}),
there exists an eigenvector $L_Y \in \PL^k(Y)$ of $g^*|_{\N^k(Y)}$ such that
$g^* L_Y \num d_k(g|_Y) L_Y$ in $\N^k(Y)$.
We observe that $\pi^* \colon \N^k(Y) \to \N^k(X)$ is injective and hence $L_X \coloneqq \pi^*L_Y \not\num 0$ in $\N^k(X)$; see \cite[Remark~2.13]{FL17} for the surjectivity of $\pi_*$.
Then by the projection formula, we have
$$g^*L_X = g^*\pi^*L_Y = \pi^*g^*L_Y \num d_k(g|_Y) \pi^*L_Y = d_k(g|_Y) L_X \text{ in } \N^k(X).$$
It follows readily that $d_k(g|_X) \ge d_k(g|_Y)$.

Thanks to Lemma~\ref{lemma-alt-def-deg}, the second part of the lemma follows from the fact that
the pullback of any nef and big divisor under a generically finite morphism is still nef and big.
More precisely, we choose a nef and big divisor (e.g., an ample divisor) $D_Y$ on $Y$.
Then so is $D_X\coloneqq \pi^*D_Y$ on $X$.
Further, we have
\begin{align*}
d_k(g|_X) &= \lim_{m\to \infty} \Big((g^m)^*D_X^k \cdot D_X^{n-k}\Big)^{1/m} = \lim_{m\to \infty} \Big((g^m)^*(\pi^*D_Y)^k \cdot (\pi^*D_Y)^{n-k}\Big)^{1/m} \\
&= \lim_{m\to \infty} \Big(\pi^*\big((g^m)^*D_Y^k \cdot D_Y^{n-k} \big)\Big)^{1/m} = \lim_{m\to \infty} \Big((g^m)^*D_Y^k \cdot D_Y^{n-k}\Big)^{1/m} = d_k(g|_Y).
\end{align*}
We thus prove the lemma.
\end{proof}

\subsection{Higher-dimensional Hodge index theorem} \label{subsec-HIT}

Let $X$ be a projective variety of dimension $n \ge 2$,
and $H_1, \ldots, H_{n-1}$ ample $\bR$-divisors on $X$.
We define a symmetric form $q_X$ on $\N^1(X)$ by
$$q_X(D_1, D_2) \coloneqq -D_1\cdot D_2 \cdot H_1 \cdots H_{n-2}.$$
Let
$$\Sigma(H_1 \cdots H_{n-1}) \coloneqq \{D \in \N^1(X) : D \cdot H_1 \cdots H_{n-1} = 0 \},$$
which is a hyperplane in $\N^1(X)$. Sometimes, it is also denoted as $\Sigma$ for short.

The following higher-dimensional Hodge index theorem is the main result of this subsection.\footnote{We remark that in the proof of Theorem~\ref{theorem-Tits} or Lemma~\ref{lemma-2.3I}, we only need the positive semi-definiteness of the quadratic form $q_X$.}
With the help of the Bertini-type result (i.e., Lemma~\ref{lemma-Bertini}), we provide a characteristic-free proof.
This thus generalizes \cite[Lemma~3.2]{Zhang16-TAMS} to arbitrary characteristic.

\begin{proposition} \label{prop-HIT-R}
Let $X$ be a projective variety of dimension $n \ge 2$.
Let $H_1, \ldots, H_{n-1}$ be ample $\bR$-divisors.
Then the quadratic form $q_X$ is positive definite on $\Sigma(H_1 \cdots H_{n-1})$.
In particular, let $D$ be an $\bR$-Cartier divisor such that $D \cdot H_1 \cdots H_{n-1} = D^2 \cdot H_1 \cdots H_{n-2} = 0$. Then $D \num 0$ (numerically).
\end{proposition}

We first prove Proposition~\ref{prop-HIT-R} in the case that $H_1, \ldots, H_{n-1}$ are ample $\bQ$-divisors
(although it has been well-known in literature already; see e.g. \cite[Expos\'e~XIII, Corollaire~7.4]{SGA6}).

\begin{lemma} \label{lemma-HIT-Q}
Let $X$ be a projective variety of dimension $n \ge 2$.
Let $H_1, \ldots, H_{n-1}$ be ample $\bQ$-divisors.
Then the quadratic form $q_X$ is positive definite on $\Sigma(H_1 \cdots H_{n-1})$.
\end{lemma}

\begin{proof}
We shall prove the lemma by the induction on $n$.
When $n = 2$, it is the usual Hodge index theorem for surfaces.
Suppose that $n \ge 3$ and the lemma holds in dimension $n - 1$.
It suffices to show that so is the case in dimension $n$.
Namely, for any $\bR$-Cartier divisor $D$ on $X$ such that $D \cdot H_1 \cdots H_{n-1} = 0$, then either $D^2 \cdot H_1 \cdots H_{n-2} < 0$, or $D \num 0$.
Suppose that $D^2 \cdot H_1 \cdots H_{n-2} \ge 0$.
Let $C$ be any irreducible curve on $X$.
We shall prove that $D \cdot C = 0$.
Replacing $X$ by its normalization $\nu\colon X^\nu \to X$ which is a finite morphism, we may assume that $X$ is normal by the projection formula.
Then by Lemma~\ref{lemma-Bertini}, there exists a hypersurface section $H'$ in $|dH_{n-2}|$ with $d\gg 0$ such that $H'$ is a projective variety containing $C$ (since the curve $C$ has codimension at least $2$ in $X$).
Now we have
\begin{align*}
0 &= D \cdot H_1 \cdots H_{n-3} \cdot dH_{n-2} \cdot H_{n-1} \\
&= D \cdot H_1 \cdots H_{n-3} \cdot H' \cdot H_{n-1} \\
&= D|_{H'} \cdot H_1|_{H'} \cdots H_{n-3}|_{H'} \cdot H_{n-1}|_{H'}.
\end{align*}
However, by the induction hypothesis, $q_{H'}$ is positive definite on $\Sigma(H_1|_{H'} \cdots H_{n-3}|_{H'} \cdot H_{n-1}|_{H'})$.
This yields that either
$$D^2 \cdot H_1 \cdots H_{n-3} \cdot dH_{n-2} = D^2 \cdot H_1 \cdots H_{n-3} \cdot H' = (D|_{H'})^2 \cdot H_1|_{H'} \cdots H_{n-3}|_{H'} < 0,$$
or $D|_{H'} \num 0$ (numerical equivalence for $H'$).
The first case does not happen by our previous assumption.
While the last (numerical) equality implies that $D \cdot C = D|_{H'} \cdot C = 0$, since $H'$ contains the curve $C$.
We thus complete the proof of Lemma~\ref{lemma-HIT-Q}.
\end{proof}

\begin{proof}[Proof of Proposition \ref{prop-HIT-R}]
By passing to the limit in Lemma~\ref{lemma-HIT-Q}, the quadratic form $q_X$ is positive semi-definite on $\Sigma$, i.e.,
for any $\bR$-Cartier divisor $D$ on $X$ such that $D \cdot H_1 \cdots H_{n-1} = 0$, then either $D^2 \cdot H_1 \cdots H_{n-2} \le 0$, or $D \num 0$.
Therefore, we only need to show (by the induction on $\dim X$) that if the case $D^2 \cdot H_1 \cdots H_{n-2} = 0$ happens, then $D \num 0$.

When $X$ is a surface, this is true by the usual Hodge index theorem.
Suppose that this assertion holds in dimension $n-1$.
Replacing $X$ by its normalization $\nu\colon X^\nu \to X$, we may assume that $X$ is normal by the projection formula.
Suppose to the contrary that $D \not\equiv 0$.
Then there exists an irreducible curve $C$ on $X$ such that $D \cdot C \neq 0$.
On the other hand, since $q_X$ is positive semi-definite on $\Sigma$, by the Cauchy--Schwarz inequality, for any $\bR$-Cartier divisor $D'$ whose class is in $\Sigma$, we have
$$\abs{q_X(D, D')}^2 \le q_X(D, D) \cdot q_X(D', D') = 0.$$
Hence $q_X(D, D') = 0$.
We also have $q_X(D, H_{n-1}) = 0$ since $D \in \Sigma$.
Note that the class of the ample $\bR$-divisor $H_{n-1}$ is not in $\Sigma$.
Thus $q_X(D, D') = 0$ for any $\bR$-Cartier divisor $D'$ since $H_{n-1}$ and $\Sigma$ span the whole $\N^1(X)$.

Write $H_{n-2} = \sum a_i A_i$ with $a_i \in \bR_{>0}$ and ample prime divisors $A_i$.
Then for each $A_i$, by Lemma~\ref{lemma-Bertini} we may assume that it is a hypersurface section containing $C$ (after replacing $H_{n-2}$ by some multiple).
By the previous discussion, we have $q_X(D, A_i) = 0$ for any $i$.
Hence
$$D|_{A_i} \cdot H_1|_{A_i} \cdots H_{n-2}|_{A_i} = D \cdot A_i \cdot H_1 \cdots H_{n-2} = 0.$$
Thus it follows from the induction hypothesis that for each $i$ either
$$D^2 \cdot H_1 \cdots H_{n-3} \cdot A_i = (D|_{A_i})^2 \cdot H_1|_{A_i} \cdots H_{n-3}|_{A_i} < 0,$$
or $D|_{A_i} \num 0$.
But the latter case cannot happen since $D|_{A_i} \cdot C = D \cdot C \neq 0$.
Therefore,
$$D^2 \cdot H_1 \cdots H_{n-3} \cdot H_{n-2} = D^2 \cdot H_1 \cdots H_{n-3} \cdot \sum a_i A_i < 0.$$
This contradicts with $D^2 \cdot H_1 \cdots H_{n-2} = 0$.
We have completed the proof of Proposition~\ref{prop-HIT-R}.
\end{proof}

As a corollary, we obtain the log-concavity of dynamical degrees.

\begin{corollary} \label{cor-log-concave}
The function $k \mapsto \log d_k(g)$ is concave in $k$. Namely,
$$d_k^2(g) \ge d_{k-1}(g) d_{k+1}(g) \, \text{ for any } \, 1\le k\le n-1.$$
In particular, $d_k(g) \le d_1^k(g)$ for any $0\le k\le n$, and there are two integers $0\le r\le s \le n$ such that
$$1 = d_0(g) < \cdots < d_r(g) = \cdots = d_s(g) > \cdots > d_n(g) = 1.$$
\end{corollary}

\begin{proof}
It follows readily from Proposition~\ref{prop-HIT-R} by choosing $H_1 = \cdots = H_{k-1} = (g^m)^*H$ and $H_k = \cdots = H_{n-1} = H$, where $H$ is an ample divisor.
More precisely, we have
$$q_X\Big((g^m)^*H, H\Big)^2 \ge q_X(H, H) \cdot q_X\Big((g^m)^*H, (g^m)^*H\Big).$$
See also \cite[Corollary~2.2 and Proposition~3.6]{Dinh12} for related results.
\end{proof}

\subsection{Quasi-nef sequences} \label{subsec-quasi-nef}

The notion of quasi-nef sequences was initially introduced by Zhang \cite{Zhang09}.

\begin{definition} \label{def-quasi-nef-seq}
Let $X$ be a projective variety of dimension $n\ge 2$.
Given a positive integer $s\le n$.
A {\it quasi-nef sequence of length $s$} consists of non-zero divisor classes $D_1, \ldots, D_s \in \N^1(X)$ satisfying:
\begin{enumerate}[(i)]
\item $D_1$ is nef, and
\item for every $2\le r \le s$, there are nef classes $D_{r, k}\in \Nef(X)$ such that
$$0 \not \num D_1 \cdots D_r \num \lim_{k \to \infty} D_1 \cdots D_{r-1} \cdot D_{r, k} \ \text{ in } \N^{r}(X).$$
\end{enumerate}
\end{definition}

It should be noted that $D_1 \cdots D_r$ is actually contained in the nef cone $\Nef^r(X)$ for each $r$;
see \S\ref{subsubsec-pos-cone} the definition of nef cone $\Nef^r(X)$.
The following two lemmas give us key properties of quasi-nef sequences which play an essential role in proving our main theorems.

\begin{lemma} \label{lemma-2.3I}
Let $D_1, \dots, D_s$ be a quasi-nef sequence of length $s\le n$, and $H_{s+1}, \dots, H_n$ ample $\bR$-divisors.
Then the following assertions hold.
\begin{enumerate}[{\em (1)}]
\item $D_1 \cdots D_s \cdot H_{s+1} \cdots H_n > 0$.
In particular, $D_1 \cdots D_s \not \wnum 0$.
\item Suppose that $s\le n-1$.
Then the symmetric form
$$q_X(M_1, M_2) \coloneqq -M_1\cdot M_2 \cdot D_1 \cdots D_{s-1} \cdot H_{s+1} \cdots H_{n-1}$$
is positive semi-definite on the hyperplane $\Sigma \coloneqq \Sigma(D_1 \cdots D_{s-1} \cdot H_{s+1} \cdots H_{n})$.
\end{enumerate}
\end{lemma}

\begin{proof}
(1) It follows readily from Lemma~\ref{lemma-num-weak-num} by the nefness of $D_1 \cdots D_s$.

(2) The case $s = 1$ follows directly from Proposition~\ref{prop-HIT-R}.
By passing to the limit, we obtain the positive semi-definiteness of the symmetric form $q_X$ defined by $n-2$ classes $D_1, \dots, D_{s-1}, H_{s+1}, \dots, H_{n-1}$ on $\Sigma$.
Note that $D_1, \dots, D_{s-1}$ is a quasi-nef sequence of length $s-1$ and hence $D_1 \cdots D_{s-1} \cdot H_{s+1} \cdots H_{n} \not \num 0$ by the previous assertion (1).
Thus $\Sigma$ is indeed a hyperplane in $\N^1(X)$.
\end{proof}

\begin{lemma}[{cf. \cite[Corollaire~3.5]{DS04}}] \label{lemma-2.3II}
Let $D_1, \dots, D_{s-1}, D_s$ and $D_1, \dots, D_{s-1}, D'_s$ be two quasi-nef sequences of length $s\le n-1$. Then the following assertions hold.
\begin{enumerate}[{\em (1)}]
\item Suppose that $D_1 \cdots D_s \cdot D'_s \wnum 0$. Then there exists a unique non-zero real number $b$ such that $D_1 \cdots D_{s-1} \cdot (D_s + b D'_s) \wnum 0$.
\item Suppose that $g\colon X \to X$ is an automorphism such that
$$g^*(D_1 \cdots D_{s-1} \cdot D_s) \wnum \lambda D_1 \cdots D_s \text{ and } g^*(D_1 \cdots D_{s-1} \cdot D'_s) \wnum \lambda' D_1 \cdots D_{s-1} \cdot D'_s$$
with positive real numbers $\lambda \ne \lambda'$.
Then $D_1 \cdots D_{s-1} \cdot D_s \cdot D'_s \not\wnum 0$.
\end{enumerate}
\end{lemma}

\begin{proof}
(1) We follow the proof of \cite[Corollaire~3.5]{DS04}.
We may assume that $D_s$ and $D'_s$ are not collinear in $\N^1(X)$.
Denote $F$ the plane spanned by $D_s$ and $D'_s$. Then $F$ is an $\bR$-vector subspace of $\N^1(X)$ of dimension $2$.
Let $H_{s+1}, \dots, H_{n-1}$ be ample $\bR$-divisors.
We claim that there exists a non-zero class $\widetilde D \in F$, unique up to a multiple scale, such that
\begin{equation} \label{eqn-DS3.5}
\widetilde D \cdot D_1 \cdots D_{s-1} \cdot H_{s+1} \cdots H_{n-1} \wnum 0 \ \text{ in } \W^{n-1}(X).
\end{equation}

We first fix an ample $\bR$-divisor $H_n$.
Then the symmetric form $q_X$ defined in Lemma~\ref{lemma-2.3I}(2) is positive semi-definite on the hyperplane $\Sigma$.
Note that by Lemma~\ref{lemma-2.3I}(1), we have $q_X(D_s, D_s) \le 0$ and $q_X(D'_s, D'_s) \le 0$.
We also have $q_X(D_s, D'_s) = 0$ by the assumption.
Hence the form $q_X$ is negative semi-definite on $F$.
We then deduce that $q_X$ vanishes on $F \cap \Sigma$.
The intersection $F \cap \Sigma$ is not zero by the dimension reasoning.
Hence there exists a non-zero class $\widetilde D \in F \cap \Sigma$ such that $q_X(\widetilde D, \widetilde D) = 0$.

Since $q_X$ is positive semi-definite on $\Sigma$, by the Cauchy--Schwarz inequality, for any $D \in \Sigma$, we have
$$\abs{q_X(\widetilde D, D)}^2\le q_X(\widetilde D, \widetilde D) \cdot q_X(D, D)=0.$$
Hence $q_X(\widetilde D, D)=0$.
On the other hand, $q_X(\widetilde D, H_n) = 0$ since $\widetilde D \in \Sigma$.
Note that the ample $\bR$-divisor $H_n$ is not $D_1 \cdots D_{s-1} \cdot H_{s+1} \cdots H_n$-primitive by Lemma~\ref{lemma-2.3I}(1), i.e., $H_n\notin \Sigma$. So $q_X(\widetilde D, D)=0$ for any $D\in \N^1(X)$ since $H_n$ and $\Sigma$ span the whole $\N^1(X)$.
Therefore, the equation \eqref{eqn-DS3.5} follows.

We show the uniqueness of $\widetilde D$.
If $\widetilde D$ is not unique up to a multiple scale, then the last equality is true for any $\widetilde D\in F$ since $\dim F = 2$.
In particular, it is true for $\widetilde D = D_s$.
This contradicts Lemma~\ref{lemma-2.3I}(1).
In particular, if we write $\widetilde D = a D_s + b D'_s$, then $ab \ne 0$.

It has been shown that the set of all classes $\widetilde D\in F$ satisfying the equation \eqref{eqn-DS3.5}
is equal to the straight line $F\cap\Sigma$ with $\Sigma$ defined by $D_1, \dots, D_{s-1}, H_{s+1}, \dots, H_n$.
We then deduce that $F\cap\Sigma$ does not depend on $H_n$ because the equation \eqref{eqn-DS3.5} does not depend on $H_n$. By symmetry, this intersection does not depend either on $H_{s+1}, \dots, H_{n-1}$ so that the equation \eqref{eqn-DS3.5} holds for any ample $\bR$-divisors $H_{s+1}, \dots, H_{n-1}$.
So we have proved the assertion (1).

(2) Suppose that $D_1 \cdots D_{s-1} \cdot D_s \cdot D'_s \wnum 0$.
Then by the assertion (1), there exists a unique non-zero real number $b$ such that $D_1 \cdots D_{s-1} \cdot (D_s + b D'_s) \wnum 0$.
Let $g^*$ act on the last equality.
We have $D_1 \cdots D_{s-1} \cdot (\lambda D_s + b \lambda' D'_s) \wnum 0$.
In other words, $D_1 \cdots D_{s-1} \cdot (D_s + \frac{b \lambda'}{\lambda} D'_s) \wnum 0$ with $\frac{b \lambda'}{\lambda} \ne b$.
This contradicts with the uniqueness of the non-zero real number $b$.
\end{proof}

To conclude this section, we quote the following Birkhoff's generalization of the Perron--Frobenius theorem since we frequently use it, as well as a theorem of Lie--Kolchin type for a cone.

\begin{theorem}[{cf. \cite{Birkhoff67}}] \label{theorem-Birkhoff-cone}
Let $V$ be a finite-dimensional $\bR$-vector space and $C \subset V$ a salient closed convex cone such that $C$ spans $V$ as a vector space.
Let $g$ be an $\bR$-linear endomorphism of $V$ such that $g(C) \subseteq C$.
Then there is an eigenvector $\upsilon_g \in C$ such that $g(\upsilon_g) = \rho(g) \upsilon_g$, where $\rho(g)$ is the spectral radius of $g$.
\end{theorem}

\begin{theorem}[{cf. \cite[Theorem~1.1]{KOZ09}}] \label{theorem-KOZ-cone}
Let $V$ be a finite-dimensional $\bR$-vector space and $\{0\} \ne C \subset V$ a salient closed convex cone.
Suppose that a solvable subgroup $G \le \GL(V)$ is Z-connected (i.e., its Zariski closure in $\GL(V_\bC)$ is connected with respect to the Zariski topology), and $G(C) \subseteq C$.
Then $G$ has a common eigenvector in the cone $C$.
\end{theorem}


\section{Proof of Theorem~\ref{theorem-Tits}} \label{sec-proof-Tits}


\begin{proof}[Proof of Theorem~\ref{theorem-Tits}(\ref{theorem-Tits-ass1})]
The proof will be split into the following four steps.

\begin{step} \label{proof-Tits-S1}
We first consider the induced action of $G$ on the complexification $\NS_\bC(X)$ of the N\'eron--Severi group $\NS(X)$ by the natural pullback.
Denote the induced group by $G|_{\NS_\bC(X)}$ which is a subgroup of $\GL(\NS_\bC(X))$.
By our assumption on $G$ and the classical Tits alternative theorem \cite[Theorem~1]{Tits72}, $G|_{\NS_\bC(X)}$ has a solvable subgroup of finite index.
Let $G_1$ be the inverse image, via the natural group homomorphism $G \to \GL(\NS_\bC(X))$, of the neutral component of the Zariski closure of that solvable subgroup in $\GL(\NS_\bC(X))$.
Then $G_1$ is a finite-index subgroup of $G$ such that $G_1|_{\NS_\bC(X)}$ is solvable and Z-connected in $\GL(\NS_\bC(X))$; see also \cite[Remark~3.10]{DHZ15}.

In what follows, for simplicity of notation, we will replace $G$ by the above $G_1$ and assume that $G|_{\NS_\bC(X)}$ is solvable and Z-connected.
\end{step}

\begin{step} \label{proof-Tits-S2}
Let us consider the induced action of $G$ on the finite-dimensional $\bR$-vector space $\N^1(X) \isom \NS_\bR(X)$.
Note that the natural map $G|_{\N^1(X)} \to G|_{\NS_\bC(X)}$ is an isomorphism and hence $G|_{\N^1(X)}$ is solvable and Z-connected.
Since $G|_{\N^1(X)}$ preserves the nef cone $\Nef(X)$ of $X$,
by applying Theorem~\ref{theorem-KOZ-cone} to the triplet $(\Nef(X), \N^1(X), G|_{\N^1(X)})$,
there is a common eigenvector $D_1\in \Nef(X)$ of $G|_{\N^1(X)}$.
It thus defines a group character $\chi_1\colon G \to \bR_{>0}$ satisfying $g^*D_1 \num \chi_1(g) D_1$ in $\N^1(X)$ for any $g\in G$.

We are going to produce a quasi-nef sequence beginning from $D_1$ as follows.
Consider the induced pullback action of $G$ on the finite-dimensional vector subspace
$D_1 \cdot \N^1(X) \subseteq \N^2(X)$
which is the image of $\N^1(X)$ under the linear map $\sigma_1$ induced from the multiplication by $D_1$.
It is easy to see that the $\sigma_1$-image of $\Nef(X)$ in $\N^2(X)$ is still a salient convex cone,
which is non-zero by Lemma~\ref{lemma-2.3I}(1).
Thus its closure\footnote{It is necessary to take the closure since the linear image of a closed convex cone may {\it not} be closed any more.}
$\overline{\sigma_1(\Nef(X))}$ in $\N^2(X)$ is a salient closed convex cone preserved by $G|_{\sigma_1(\N^1(X))}$.
Also, as the image of $G|_{\N^1(X)}$, the induced group $G|_{\sigma_1(\N^1(X))}$ is solvable and Z-connected.
Therefore, applying Theorem~\ref{theorem-KOZ-cone} to the triplet $(\overline{\sigma_1(\Nef(X))}, \sigma_1(\N^1(X)), G|_{\sigma_1(\N^1(X))})$,
there is a common eigenvector $\sigma_1(D_2) \in \overline{\sigma_1(\Nef(X))}$ of $G|_{\sigma_1(\N^1(X))}$.
We then define a group character $\chi_2\colon G \to \bR_{>0}$ such that $g^*(D_1 \cdot D_2) \num \chi_1(g) \chi_2(g) D_1 \cdot D_2$ in $\N^2(X)$ for any $g\in G$.

Repeating this procedure, we will obtain a quasi-nef sequence $D_1, \dots, D_n$ of length $n$ such that
for any $1\le r\le n-1$, the induced group $G|_{\sigma_r(\N^1(X))}$ is solvable, Z-connected,
and preserves the salient closed convex cone
$$\overline{\sigma_r(\Nef(X))} \subset \sigma_r(\N^1(X)) \subseteq \N^{r+1}(X).$$
Moreover, there are group characters $\chi_i\colon G \to \bR_{>0}$ with $1\le i\le n$ such that for any $1\le t \le n$,
$$g^*(D_1 \cdots D_t) \num \chi_1(g) \cdots \chi_t(g) D_1 \cdots D_t \ \text{ in } \N^t(X).$$
Also, note that $\chi_1(g) \cdots \chi_n(g) = 1$ for any automorphism $g$ in $G$.
We then define a group homomorphism of $G$ as follows:
\begin{equation} \label{eqn-DS-morphism}
\psi\colon G \to (\bR^{n-1}, +), \quad g\mapsto (\log\chi_1(g), \dots, \log\chi_{n-1}(g)).
\end{equation}
\end{step}

\begin{step} \label{claim-Ker=N}
We claim that $\Ker \psi = N(G)$; in particular, $N(G) \unlhd G$.
On the one hand, for any $g \in N(G)$, it follows from Corollary~\ref{cor-log-concave} that all dynamical degrees equal $1$ and hence $\chi_1(g) \cdots \chi_t(g) \le 1$ for any $1\le t\le n$.
Observe that the pullback action of $g^*$ on $N^t(X)$ is defined over $\bZ$ so that all minimal polynomials of dynamical degrees are also defined over $\bZ$.
We thus have $\chi_1(g) \cdots \chi_t(g) = 1$ for any $t$.
This yields that $\chi_t(g) = 1$ for all $t$, i.e., $g\in \Ker\psi$.

On the other hand, suppose that there is an automorphism $g_{\mathrm o} \in \Ker \psi$ but $g_{\mathrm o} \notin N(G)$.
Namely, $\chi_i(g_{\mathrm o}) = 1$ for all $1\le i\le n$, but the first dynamical degree $d_1(g_{\mathrm o}) > 1$.
Applying the generalized Perron--Frobenious theorem to the triplet $(\Nef(X), \N^1(X), g_{\mathrm o}^*|_{\N^1(X)})$ (see Theorem~\ref{theorem-Birkhoff-cone}),
there is a nef eigenvector $L_{g_{\mathrm o}}$ of $g_{\mathrm o}^*|_{\N^1(X)}$ such that $g_{\mathrm o}^*L_{g_{\mathrm o}} \num d_1(g_{\mathrm o}) L_{g_{\mathrm o}}$ in $\N^1(X)$.
Since $d_1(g_{\mathrm o}) > 1$, we have $D_1 \cdots D_{n-1} \cdot L_{g_{\mathrm o}} = 0$ by the following displayed equalities
$$D_1 \cdots D_{n-1} \cdot L_{g_{\mathrm o}} = g_{\mathrm o}^*(D_1 \cdots D_{n-1} \cdot L_{g_{\mathrm o}}) = d_1(g_{\mathrm o}) D_1 \cdots D_{n-1} \cdot L_{g_{\mathrm o}}.$$
Note that for each $s$, $D_1, \dots, D_{s}$ is a quasi-nef sequence of length $s$ by the construction of $D_i$ in Step~\ref{proof-Tits-S2}. \\
Then $D_1, \dots, D_{n-2}, L_{g_{\mathrm o}}$ cannot be a quasi-nef sequence of length $n-1$
(which implies that $D_1 \cdots D_{n-2} \cdot L_{g_{\mathrm o}} \num 0$).
Since otherwise, by applying Lemma~\ref{lemma-2.3II}(2) to the two quasi-nef sequences $D_1, \dots, D_{n-1}$ and $D_1, \dots, D_{n-2}, L_{g_{\mathrm o}}$,
we would have $D_1 \cdots D_{n-2} \cdot D_{n-1} \cdot L_{g_{\mathrm o}} \ne 0$.
This is a contradiction.
Repeating this argument several times, we eventually show that $D_1 \cdot L_{g_{\mathrm o}} \num 0$.
Then by Lemma~\ref{lemma-2.3II}(1), $D_1 + b L_{g_{\mathrm o}} \wnum 0$ for a unique real number $b$.
Let $g_{\mathrm o}^*$ act on the last equality.
We have $D_1 + bd_1(g_{\mathrm o}) L_{g_{\mathrm o}} \wnum 0$ with $bd_1(g_{\mathrm o}) \ne b$.
This contradicts with the uniqueness of the real number $b$, and hence the claim follows.
\end{step}

\begin{step} \label{claim-discrete-image}
We claim that $\im \psi$ is discrete in the additive group $\bR^{n-1}$ and hence $\im \psi \isom \bZ^{\oplus r}$ with $r\le n-1$.
It suffices to show that the origin $\mathrm{o}$ is an isolated point in $\im \psi$.
That is, there exists a constant $0 < \eps \ll 1$ such that for any $g\in G$, $\psi(g) \not\in B_\eps(\mathrm{o}) \setminus \{\mathrm{o}\}$, where $\mathrm o$ is the usual origin of the Euclidean space $\bR^{n-1}$.

Choose $\eps = \frac{\log A}{n-1}$, where $A > 1$ is a constant depending only on the Picard number $\rho(X)\coloneqq \rank \NS(X)$ of $X$ such that for any automorphism $g \in \Aut(X)$,
$$\text{if } d_1(g) > 1 \text{ then } d_1(g) \ge A.$$
The existence of such a constant $A$ is due to the fact that the minimal polynomial of $g^*|_{\N^1(X)}$ is a monic polynomial of degree $\rho(X)$ with integral coefficients (see also the proof of \cite[Corollary~3.7]{Dinh12}).
Suppose that there is an automorphism $g_{\mathrm o}\in G$ such that $\psi(g_{\mathrm o}) \in B_\eps(\mathrm{o}) \setminus \{\mathrm{o}\}$,
i.e., $0 < \norm{\psi(g_{\mathrm o})} < \eps$,
where $\norm{\cdot}$ is the standard Euclidean norm of $\bR^{n-1}$.
Then $e^{-\eps} < \chi_i(g_{\mathrm o}) < e^\eps$ for all $i$.
Also, by Step~\ref{claim-Ker=N}, $g_{\mathrm o}$ is of positive entropy since $g_{\mathrm o} \notin \Ker \psi$.
Thus by Theorem~\ref{theorem-Birkhoff-cone} again,
there is a nef eigenvector $L_{g_{\mathrm o}}$ of $g_{\mathrm o}^*|_{\N^1(X)}$ such that $g_{\mathrm o}^*L_{g_{\mathrm o}} \num d_1(g_{\mathrm o}) L_{g_{\mathrm o}}$ in $\N^1(X)$ with $d_1(g_{\mathrm o}) > 1$.
If $D_1 \cdots D_{n-1} \cdot L_{g_{\mathrm o}} \ne 0$, then $d_1(g_{\mathrm o}) = (\chi_1(g_{\mathrm o}) \cdots \chi_{n-1}(g_{\mathrm o}))^{-1} < e^{(n-1)\eps}$ because of the following equalities
$$D_1 \cdots D_{n-1} \cdot L_{g_{\mathrm o}} = g_{\mathrm o}^*(D_1 \cdots D_{n-1} \cdot L_{g_{\mathrm o}}) = \chi_1(g_{\mathrm o}) \cdots \chi_{n-1}(g_{\mathrm o}) \cdot d_1(g_{\mathrm o}) D_1 \cdots D_{n-1} \cdot L_{g_{\mathrm o}}.$$
If $D_1 \cdots D_{n-1} \cdot L_{g_{\mathrm o}} = 0$, as the proof in Step~\ref{claim-Ker=N} (essentially using Lemma~\ref{lemma-2.3II}),
we would have $d_1(g_{\mathrm o}) = \chi_i(g_{\mathrm o}) < e^\eps$ for some $1\le i\le n-1$.
Therefore, in either case, we prove that
$$d_1(g_{\mathrm o}) < e^{(n-1)\eps} = A.$$
This is a contradiction of the choice of the constant $A$ and hence the claim follows.
\end{step}

We thus complete the proof of Theorem~\ref{theorem-Tits}(\ref{theorem-Tits-ass1}).
\end{proof}

\begin{proof}[Proof of Theorem~\ref{theorem-Tits}(\ref{theorem-Tits-ass2})]
Note that $G|_{\NS_\bC(X)}$ does not contain any non-abelian free subgroup because neither does $G$.
So according to the classical Tits alternative theorem \cite[Theorem~1]{Tits72}, $G|_{\NS_\bC(X)}$ is virtually solvable.
Let $K$ denote the kernel of the natural group homomorphism $G\to \GL(\NS_\bC(X))$.
Then by \cite[Lemma~5.5]{Dinh12}, it suffices to prove that $K$ is virtually solvable.
Denote $K^0 \coloneqq K\cap \Aut^0(X)$.
By a Fujiki--Lieberman type theorem (see \cite[Remark~2.6]{MZ18} for a proof in arbitrary characteristic using a Hilbert scheme argument), we have
$$\left[K : K^0\right] = \left[K\cdot \Aut^0(X) : \Aut^0(X)\right] \le \left[\Aut_{[L]}(X) : \Aut^0(X)\right] < \infty,$$
where $\Aut_{[L]}(X)$ denotes the subgroup of $\Aut(X)$ fixing an ample class $[L] \in \NS(X)$.
Also note that $K^0 = G^0$ is finitely generated.
Hence, after replacing $K$ by $K^0$, we may assume that $K$ is contained in $\Aut^0(X)$ and finitely generated.
Let $\bK$ be the Zariski closure of $K$ in the smooth algebraic group $(\aut_X^0)_{\red}$.
Then there exists a normal subgroup scheme $\bH$ of $\bK$ such that the quotient $\bK/\bH$ is affine and $\bH$ is commutative (cf.~\cite[Theorem~1]{Brion17}).
Applying the Tits alternative theorem to the finitely generated linear group $K/K\cap \bH(\bk) \le (\bK/\bH)(\bk)$ (cf.~\cite[Corollary~1]{Tits72}), it follows that either $K/K\cap \bH(\bk)$ (and hence $K$) contains a non-abelian free subgroup, or $K/K\cap \bH(\bk)$ is virtually solvable.
The first case cannot happen since we assume that $G$ does not contain any non-abelian free subgroup.
Hence by \cite[Lemma~5.5]{Dinh12} again, $K$ is virtually solvable which concludes the proof of the assertion (\ref{theorem-Tits-ass2}).
\end{proof}

The following example shows that the upper bound $\dr(G) \le n-1$ is optimal.
We refer to \cite{Hu-ecdg} for a systematic investigation of dynamical degrees on general abelian varieties.

\begin{example}
\label{eg:maximal-rank}
Let $E$ be an elliptic curve such that its endomorphism algebra $\End^0(E) \coloneqq \End(E) \otimes_\bZ \bQ = \bQ$
(in positive characteristic, this is equivalent to saying that $E$ cannot be defined over a finite field; see \cite[\S22, Deuring's Theorem]{Mumford}).
Let $X = E^n$.
Then by Poincar\'e's complete reducibility theorem (cf. \cite[\S19, Theorem~1]{Mumford}),
we have $\End^0(X) = \Mat_n(\End^0(E))$,
where $\Mat_n(R)$ denotes the ring of $n\times n$ matrices over $R$.
In particular, there is still a faithful $\SL_n(\bZ)$-action on $X$ as in \cite[Example~4.5]{DS04} or \cite[Example~1.4]{Dinh12}.
Hence it suffices to show that if $g = A \in \SL_n(\bZ)$ has the spectral radius $\rho(A) > 1$, then $g$ is of positive entropy in the sense that $\rho(g^*|_{\NS_\bR(X)}) > 1$.

We first notice that the N\'eron--Severi space $\NS_\bQ(X)$ can be identified with the subgroup of $\End^0(X)$ consisting of symmetric elements with respect to the Rosati involution (cf. \cite[\S21]{Mumford}).
More precisely, we send a line bundle $\sL$ on $X$ to $\phi_{\sL_0}^{-1}\circ \phi_\sL$, where $\sL_0$ is a fixed ample line bundle on $X$ and in general, the induced homomorphism $\phi_\sL \colon X \to \widehat X \coloneqq \Pic^0(X)$ is defined by $x \mapsto T_x^* \sL \otimes \sL^{-1}$.
Hence the pullback action $g^*$ on $\NS_\bQ(X)$ can be reinterpreted as follows:
$$\psi \colon \NS_\bQ(X) \lra \NS_\bQ(X), \quad \phi_{\sL_0}^{-1} \circ \phi_\sL \mapsto \phi_{\sL_0}^{-1} \circ \phi_{g^*\!\sL} .$$
Note that $\phi_{\sL_0}^{-1} \circ \phi_{g^*\!\sL} = \phi_{\sL_0}^{-1} \circ \widehat g \circ \phi_\sL \circ g = g^{\dagger} \circ \phi_{\sL_0}^{-1} \circ \phi_\sL \circ g$,
where $\widehat g$ is the induced dual automorphism of $\widehat X$ and $g^{\dagger} \coloneqq \phi_{\sL_0}^{-1} \circ \widehat g \circ \phi_{\sL_0}$ is the Rosati involution of $g$.
Thus, we can naturally extend $\psi$ to the whole endomorphism algebra $\End^0(X)$ in the following way:
$$\Psi \colon \End^0(X) \lra \End^0(X) \quad \text{via} \quad \phi \mapsto g^{\dagger} \circ \phi \circ g.$$
It is easy to verify that $\Psi$ is represented by the Kronecker product $A \otimes A$, whose spectral radius is $\rho(A)^2$.
Furthermore, in our case, the N\'eron--Severi space $\NS_\bR(X)$ is isomorphic with the subspace $\Sym_n(\bR)$ of real symmetric matrices (cf. \cite[\S21, Theorem~6]{Mumford}).
In particular, if we denote all eigenvalues of $A$ by $\lambda_i$, $1\le i\le n$, then all eigenvalues of $\Psi$ are all possible products $\lambda_i \lambda_j$ for $1\le i, j\le n$ and all eigenvalues of $\psi$ are just all $\lambda_i \lambda_j$ with $1\le i \le j\le n$.
It follows that $\rho(g^*|_{\NS_\bR(X)}) = \rho(A)^2 > 1$.
\end{example}


\section{Proof of Theorem~\ref{theorem-max-rank}} \label{sec-proof-max-rank}


\noindent
Our proof of Theorem~\ref{theorem-max-rank} will heavily rely on the following lemma (and its generalization Lemma~\ref{lemma-G-rat-fib-rank}) which improves \cite[Lemma~2.10]{Zhang09}.
It turns out that our lemmas are crucial to show the first assertion of Theorem~\ref{theorem-max-rank} or Lemma~\ref{lemma-max-rank-kappa}.
In fact, as we have mentioned in the introduction that, Dinh or Zhang's argument essentially depends on the Deligne--Nakamura--Ueno theorem (see e.g. \cite[Theorem~4.3]{Dinh12}), which is not known in positive characteristic, as far as we know.

Let $G$ be a subgroup of $\Aut(X)$. A morphism $\pi\colon X\to Y$ of projective varieties, or more generally, a rational map $\pi\colon X \ratmap Y$ is called {\it $G$-equivariant} if $G$ descends to a biregular (possibly non-faithful) action on $Y$.
The induced $G$-action on $Y$ would be denoted by $G|_Y$.

\begin{lemma} \label{lemma-G-fib-rank}
Let $\pi\colon X\to Y$ be a $G$-equivariant surjective morphism of projective varieties with $n = \dim X > \dim Y = m > 0$.
Suppose that $G|_{\N^1(X)}$ is solvable and Z-connected.
Then we have $\dr(G|_X) \le \dr(G|_Y) + n - m - 1$.
In particular, $\dr(G|_X) \le n - 2$ as $\dr(G|_Y) \le m - 1$.
\end{lemma}

\begin{proof}
It is known that $\pi^*\colon \N^1(Y) \to \N^1(X)$ is injective and $\pi^*\N^1(Y)$ is a $g^*$-invariant subspace of $\N^1(X)$ for any $g\in G$.
Then there is a natural restriction homomorphism $G|_{\N^1(X)} \surjmap G|_{\pi^*\N^1(Y)} \isom G|_{\N^1(Y)}$.
By the assumption, $G|_{\N^1(Y)}$ is also solvable and Z-connected.
Running Steps~\ref{proof-Tits-S1} and \ref{proof-Tits-S2} in the proof of Theorem~\ref{theorem-Tits} for $Y$ and $G|_Y$,
we obtain a quasi-nef sequence $B_1, \dots, B_m$ of length $m$ in $\N^1(Y)$
and group characters $\chi_i\colon G|_Y \to \bR_{>0}$ with $1\le i\le m$ such that for any $1\le t \le m$ and $g_Y \in G|_Y$,
$$g_Y^*(B_1 \cdots B_t) \num \chi_1(g_Y) \cdots \chi_t(g_Y) B_1 \cdots B_t \ \text{ in } \N^t(Y), \text{ and } \chi_1(g_Y) \cdots \chi_m(g_Y) = 1.$$
We claim that $D_1 \coloneqq \pi^*B_1, \dots, D_m \coloneqq \pi^*B_m$ is also a quasi-nef sequence of length $m$ in $\N^1(X)$.
Indeed, it suffices to show that $D_1 \cdots D_k \not\num 0$ in $N^k(X)$ for any $1\le k\le m$.
This actually follows from the fact that $\pi^* \colon N^k(Y) \to N^k(X)$ is injective.
Replacing $\chi_i$ by its composition with the natural group homomorphism $\sigma \colon G|_X \to G|_Y$,
we may assume that each $\chi_i$ is a group character of $G|_X$.
By continuing Step~\ref{proof-Tits-S2} for $X$ and $G|_X$, we can produce $D_{m+1}, \dots, D_n\in \N^1(X)$ such that $D_1, \dots, D_n$ is a quasi-nef sequence of length $n$ in $\N^1(X)$ extending $D_1, \dots, D_m$.
At the same time, we also obtain group characters $\chi_i\colon G|_X \to \bR_{>0}$ with $m+1\le i\le n$.
Together with the previous $m$ characters,
they satisfy that for any $1\le t \le n$ and $g\in G|_X$,
$$g^*(D_1 \cdots D_t) \num \chi_1(g) \cdots \chi_t(g) D_1 \cdots D_t \ \text{ in } \N^t(X),$$
and
$$\chi_1(g) \cdots \chi_m(g) = \chi_1(g) \cdots \chi_n(g) = 1.$$
We then define group homomorphisms $\psi_X \colon G|_X \to \bR^{n-1}$ and $\psi_Y \colon G|_Y \to \bR^{m-1}$ as in \eqref{eqn-DS-morphism}.

Note that the $\sigma$-image of $N(G|_X)$ is a normal subgroup of $N(G|_Y)$ since $\pi^* \colon \N^1(Y) \to \N^1(X)$ is injective.
Let $H|_X \unlhd G|_X$ denote the inverse image of $N(G|_Y)$ under $\sigma$ which contains $N(G|_X)$.
By Step~\ref{claim-Ker=N} in the proof of Theorem~\ref{theorem-Tits}, we have the following commutative diagram:
\[\xymatrix{
1 \ar[r] & H|_X/N(G|_X) \ar[r] \ar@{^{(}->}[d]_{h}^{} & G|_X/N(G|_X) \ar@{^{(}->}[d]_{\ol\psi_X}^{} \ar[r] & G|_Y/N(G|_Y) \ar@{^{(}->}[d]_{\ol\psi_Y}^{} \ar[r] & 1 \\
0 \ar[r] & \bR^{n-m} \ar[r] & \bR^{n-1} \ar[r]^{\proj} & \bR^{m-1} \ar[r] & 0,
}\]
where $\proj$ is the obvious projection of $\bR^{n-1}$ to the first $m-1$ coordinates
and the injection $h$ is just the restriction of $\ol\psi_X$ to $H|_X/N(G|_X)$.
Recall that the first $m$ group characters $\chi_1, \ldots, \chi_m$ vanish on $N(G|_Y)$ and hence vanish on $H|_X$ as well.
In particular, the first coordinate of $\im h$ (which is given by $\log \chi_m$) vanishes.
Thus, it follows from Step~\ref{claim-discrete-image} in the proof of Theorem~\ref{theorem-Tits} that $\im h$ is a discrete subgroup of $\bR^{n-m-1}$, i.e., $\im h \isom \bZ^{\oplus \ell}$ for some $\ell \le n-m-1$.
So the prescribed dynamical rank estimate follows and we complete the proof of Lemma~\ref{lemma-G-fib-rank}.
\end{proof}

On the other hand, we shall see that if the $G$-equivariant surjective morphism $\pi \colon X \to Y$ is generically finite,
then the dynamical ranks of $G|_X$ and $G|_Y$ are the same.
As a consequence, the dynamical rank $\dr(G)$ is an invariant of a $G$-equivariant birational morphism.

\begin{lemma} \label{lemma-inv-dyn-rank}
Let $\pi\colon X\to Y$ be a $G$-equivariant generically finite surjective morphism of projective varieties.
Then after replacing $G$ by a finite-index subgroup, $G|_{N^1(X)}$ is solvable and Z-connected if and only if so is $G|_{N^1(Y)}$.
In particular, by Lemma~\ref{lemma-inv-dyn-deg}, $\dr(G|_X) = \dr(G|_Y)$.
\end{lemma}

\begin{proof}
The `only if' direction is easy (see also the very beginning part of the proof of Lemma~\ref{lemma-G-fib-rank}).
Also, being Z-connected is automatically true after replacing $G$ by a finite-index subgroup.
So it suffices to show that if $G|_{N^1(Y)}$ is solvable then $G|_{N^1(X)}$ is virtually solvable.
Consider the natural restriction group homomorphism $G|_{\N^1(X)} \surjmap G|_{\pi^*\N^1(Y)} \isom G|_{\N^1(Y)}$ as in the proof of Lemma~\ref{lemma-G-fib-rank}.
Let $K|_{\N^1(X)}$ denote its kernel, where $K\le G$ is a subgroup of $G$.
Choose any ample divisor $H_Y$ on $Y$.
Then $K$ fixes the nef and big divisor $\pi^*H_Y$ (since $\pi$ is generically finite).
It follows from Lemma~\ref{lemma-alt-def-deg} that $K$ is of null entropy.
So by \cite[Theorem~2.2]{CWZ14} (in fact, its proof in \cite{CWZ14} is a linear algebraic group argument and holds in our situation as well),
$K|_{\N^1(X)}$ is virtually unipotent and hence virtually solvable.
Then it follows from \cite[Lemma~5.5]{Dinh12} that $G|_{N^1(X)}$ is virtually solvable.
Lastly, we note that replacing $G$ by a finite-index subgroup will not change the dynamical rank.
So we prove the lemma.
\end{proof}

Thanks to the lemma above, we are able to weaken the condition on $\pi$ in Lemma~\ref{lemma-G-fib-rank} to a $G$-equivariant dominant rational map.
For the sake of completeness, we give the full statement as follows.

\begin{lemma} \label{lemma-G-rat-fib-rank}
Let $\pi\colon X\ratmap Y$ be a $G$-equivariant dominant rational map of projective varieties with $n = \dim X > \dim Y = m > 0$.
Suppose that $G|_{\N^1(X)}$ is solvable and Z-connected.
Then we have $\dr(G|_X) \le \dr(G|_Y) + n - m - 1$.
In particular, $\dr(G|_X) \le n - 2$ as $\dr(G|_Y) \le m - 1$.
\end{lemma}

\begin{proof}
Let $\Gamma$ be the normalization of the main component of the closure of the graph of $\pi$ in $X \times Y$.
Then $p\colon \Gamma \to X$ is a surjective birational morphism such that $q = \pi \circ p\colon \Gamma \to Y$ is a surjective morphism.
The $G$-actions on $X$ and $Y$ naturally lift to a biregular faithful action on $\Gamma$ so that both $p$ and $q$ are $G$-equivariant.
Then the lemma follows by applying Lemma~\ref{lemma-inv-dyn-rank} to $p$ and Lemma~\ref{lemma-G-fib-rank} to $q$.
\end{proof}

Using the same argument as in the proof of Lemma~\ref{lemma-G-rat-fib-rank}, one can show that the dynamical rank $\dr(G)$ is actually an invariant of a $G$-equivariant birational map.

\begin{lemma} \label{lemma-bir-inv-dyn-rank}
Let $\pi\colon X\ratmap Y$ be a $G$-equivariant generically finite dominant rational map of projective varieties.
Then we have $\dr(G|_X) = \dr(G|_Y)$.
\qed
\end{lemma}

Now, we are ready to apply Lemma~\ref{lemma-G-rat-fib-rank} to the Iitaka fibration and obtain the dynamical rank estimate in terms of those $\kappa$'s in Theorem~\ref{theorem-max-rank}.
See \S\ref{subsubsec-kappa} for their precise definitions.

\begin{lemma} \label{lemma-max-rank-kappa}
Let $X$ be a projective variety of dimension $n\ge 2$, and $G$ a subgroup of $\Aut(X)$ such that $G|_{\N^1(X)}$ is solvable and Z-connected.
Then the following assertions hold.
\begin{enumerate}[{\em (1)}]
\item Let $\nu \colon X^\nu \to X$ be the normalization of $X$.
Then $\dr(G|_X) = \dr(G|_{X^\nu}) \le \max\{0, n-1-\kappa(\omega_{X^\nu})\}$, where $\kappa(\omega_{X^\nu})$ is the Kodaira--Iitaka dimension of $X^\nu$.
\item Let $\kappa(X)$ be the Kodaira dimension of $X$. Then $\dr(G|_X) \le \max\{0, n-1-\kappa(X)\}$.
\item Suppose that $\dr(G|_X) = n-1$. Then $\kappa(\omega_{X^\nu}) \le 0$ and $\kappa(X) \le 0$.
\end{enumerate}
\end{lemma}

\begin{proof}
It suffices to show the first assertion.
In fact, by \cite[Proposition~B.1]{Patakfalvi18}, we know that $\kappa(X^\nu) \le \kappa(\omega_{X^\nu})$.
We also have $\kappa(X) = \kappa(X^\nu)$, since Luo's $\kappa(X)$ is a birational invariant of $X$.
Then it is easy to verify that $\max\{0, n-1-\kappa(\omega_{X^\nu})\} \le \max\{0, n-1-\kappa(X)\}$.
So the assertion (2) follows directly from the first one.
The third assertion is a consequence of the first two.

Note that the normalization $\nu$ is a $G$-equivariant finite surjective morphism.
So by Lemma~\ref{lemma-inv-dyn-rank}, we have the equality $\dr(G|_X) = \dr(G|_{X^\nu})$.
To prove the assertion (1), it remains to show the inequality.
For simplicity, we may assume that $X$ itself is normal and suppress the notation $X^\nu$ from now on.
First, it is known that for sufficiently large $m$, there exists a rational map
$$\phi \coloneqq \phi_{\omega_X^{[m]}}\colon X\ratmap Y \subseteq \bP_\bk^N \coloneqq \bP H^0(X, \omega_X^{[m]}),$$
where $Y$ is the image of $\phi$ satisfying $\dim Y = \kappa(\omega_X)$.
Also, the $G$-action on $X$ descends to a linear action on $\bP_\bk^N$ via pullback of sections, so does on $Y$.
Hence $G|_Y$ is a subgroup of $\Aut(\bP_\bk^N, Y) \coloneqq \{g\in \Aut(\bP_\bk^N) : g(Y) = Y\}$.
Let $H_Y$ denote the restriction of some hyperplane $H$ on $\bP_\bk^N$ to $Y$.
Then $g^*H_Y \sim H_Y$ for any $g\in G$ and hence $G|_Y$ is of null entropy (see Lemma~\ref{lemma-equiv-dyn-deg}).
In other words, $\dr(G|_Y) = 0$.
If $\kappa(\omega_X) = n$, then $\phi$ is a $G$-equivariant birational map so that $\dr(G|_X) = \dr(G|_Y) = 0$ by Lemma~\ref{lemma-bir-inv-dyn-rank}.
It remains to consider the case that $0 < \kappa(\omega_X) < n$.
Applying Lemma~\ref{lemma-G-rat-fib-rank} to $\phi\colon X \ratmap Y$, we have $\dr(G|_X)\le \dr(G|_Y) + n - \dim Y - 1 = n - \kappa(\omega_X) - 1$.
Hence the lemma follows.
\end{proof}

\begin{remark} \label{remark-max-rank-anti-kappa}
The proof of Lemma~\ref{lemma-max-rank-kappa}(1) remains valid if we replace $\kappa(\omega_X)$ by the anti-Kodaira--Iitaka dimension $\kappa^{-1}(\omega_X)$; see \S\ref{subsubsec-kappa} for its definition.
In other words, as in the assertion (1), the dynamical rank $\dr(G|_X) = \dr(G|_{X^\nu}) \le \max\{0, n-1-\kappa^{-1}(\omega_{X^\nu})\}$.
In particular, if $\dr(G|_X) = n-1$, then $\kappa^{-1}(\omega_{X^\nu}) \le 0$.
\end{remark}

Lastly, to complete the proof of Theorem~\ref{theorem-max-rank}, we only need to apply Lemma~\ref{lemma-G-rat-fib-rank} to the Albanese morphism $\alb_X$ and the Albanese map $\biralb_X$.
See \S\ref{subsubsec-alb} for their precise definitions.

\begin{lemma} \label{lemma-max-rank-q}
Let $X$ be a projective variety of dimension $n\ge 2$, and $G$ a subgroup of $\Aut(X)$ such that the dynamical rank $\dr(G) = n-1$.
Then the following assertions hold.
\begin{enumerate}[{\em (1)}]
\item The irregularity $q(X) = 0$ or $n$.
If the latter case happens, i.e., $q(X) = n$, then the Albanese morphism $\alb_X$ is generically finite surjective.
\item The birational irregularity $\tilde q(X) = 0$ or $n$.
In the latter case, i.e., $\tilde q(X) = n$, the Albanese map $\biralb_X$ is generically finite dominant.
\end{enumerate}
\end{lemma}

\begin{proof}
We show the first assertion only since the second one is exactly the same.
It follows from the universal property of the Albanese morphism that $\alb_X \colon X \to \Alb(X)$ is $G$-equivariant and the image of $\alb_X$, denoted by $Y$, generates the Albanese variety $\Alb(X)$ of $X$.
Suppose that $q(X) = \dim \Alb(X) > 0$.
Then $\dim Y > 0$.
By applying Lemma~\ref{lemma-G-rat-fib-rank} to the $G$-equivariant morphism $X \to Y$,
we have $\dim Y = \dim X = n$ because the dynamical rank $\dr(G|_X) = n-1$.
Namely, $X$ has maximal Albanese dimension and hence $\alb_X$ is a generically finite morphism.
We then claim that $\kappa(Y) = 0$.
First, by Lemma~\ref{lemma-inv-dyn-rank}, $\dr(G|_{Y}) = \dr(G|_X) = n-1$.
It thus follows from Lemma~\ref{lemma-max-rank-kappa} that $\kappa(Y) \le 0$.
On the other hand, \cite[Theorem~3]{Abramovich94} asserts that $\kappa(Y) = \dim Y/B \ge 0$, where $B \coloneqq \Stab(Y)$ is the maximal closed subgroup of $\Alb(X)$ such that $B + Y = Y$.
Then the claim follows.
Therefore, $\dim Y/B = 0$ and hence $Y = B$ is an abelian subvariety of $\Alb(X)$.
This yields that $Y = \Alb(X)$ since $Y$ generates $\Alb(X)$.
So we have eventually showed that the Albanese morphism $\alb_X$ is surjective.
We thus deduce that $q(X) = \dim \Alb(X) = \dim Y = \dim X = n$.
\end{proof}

\begin{remark}
Let $(X, G)$ be as in Lemma~\ref{lemma-max-rank-q}.
In the case that $q(X) = n$ (and hence $\tilde q(X) = n$ automatically), we also have $\kappa(\omega_{X^\nu}) = 0$,
where $X^\nu$ denotes the normalization of $X$.
Indeed, by the maximal dynamical rank assumption, we have seen in Lemma~\ref{lemma-max-rank-kappa} that $\kappa(\omega_{X^\nu}) \le 0$.
We then consider the Albanese morphism $\alb_{X^\nu} \colon X^\nu \to \Alb(X^\nu)$ of $X^\nu$.
Note that $\nu$ is a $G$-equivariant finite surjective morphism.
Hence $q(X^\nu) \ge q(X)$ and $\dr(G|_{X^\nu}) = \dr(G|_X)$ is maximal by Lemma~\ref{lemma-inv-dyn-rank}.
Then according to Lemma~\ref{lemma-max-rank-q}(1), $\alb_{X^\nu}$ is a generically finite surjective morphism.
Applying \cite[Proposition~1.4]{HPZ17} to $\alb_{X^\nu}$, it follows that $\alb_{X^\nu}$ is separable
and hence $\kappa(\omega_{X^\nu}) \ge 0$.
Thus, we show that $\kappa(\omega_{X^\nu}) = 0$.

However, we do not know whether $\kappa(X) = 0$ (in general, it may occur that $\kappa(X) < \kappa(\omega_X)$ for singular $X$; see e.g. \cite[Example~2.1.6]{Lazarsfeld-I}).
Nevertheless, if we assume further that $X$ is smooth, then the two Kodaira dimensions agree and both of them vanish.
In this case, $\alb_X = \biralb_X$ is birational by \cite[Corollary~2]{Kawamata81} in characteristic zero and by \cite[Theorem~0.2]{HPZ17} in positive characteristic.
\end{remark}

\begin{proof}[Proof of Theorem \ref{theorem-max-rank}]
It follows readily from Theorem~\ref{theorem-Tits}, Lemmas~\ref{lemma-max-rank-kappa} and \ref{lemma-max-rank-q}.
\end{proof}


\vskip 1em
\phantomsection
\addcontentsline{toc}{section}{Acknowledgments}
\noindent \textbf{Acknowledgments. }
The author would like to thank De-Qi Zhang for many inspiring discussions and comments. He also thanks Dragos Ghioca and Zinovy Reichstein for their support and helpful comments, Mihai Fulger and Brian Lehmann for answering several questions regarding to their paper \cite{FL17}.
He is grateful to Keiji Oguiso for his help and useful suggestions, the referee for many valuable comments.
Last but not least, the author is deeply indebted to Tomohide Terasoma for providing the appendix which elaborates the reasonability of the finite generation condition in Theorem~\ref{theorem-Tits}(2).






\newpage

\appendix

\pagestyle{appendix}

\section{An example of a solvable subgroup of an automorphism group} \label{appendix}

\begin{center}
{\sc tomohide terasoma}
\end{center}

\vskip 2em

\noindent
This is an example of a finitely generated solvable subgroup of the automorphism group of a projective variety whose homological trivial part is infinitely generated.


\subsection{Free groups generated by two elements}

Let $F_2=\langle a, b\rangle$ be a free group generated by two elements $a,b$.
Let 
$$
F'_2 \coloneqq [F_2,F_2]
\text{ and }
F''_2 \coloneqq [F'_2,F'_2]
$$
be the first and second derived subgroup of $F_2$, respectively.
Then the maximal abelian quotient group 
$F_2^{\ab} \coloneqq F_2/F'_2$ is isomorphic to the free abelian group
of rank $2$ generated by the image $\bar{a}, \bar{b}$ of $a, b$.
The group $F_2$ acts on the abelian group
$(F'_2)^{\ab} \coloneqq F_2'/F_2''$ by the adjoint action 
$\ad(g) \in \Aut((F'_2)^{\ab})$ 
for $g\in F_2$ defined by
$$
\ad(g)(s)=gsg^{-1}.
$$
Since
$$
[\ad(g),\ad(h)]=\ad(g)\ad(h)
\ad(g)^{-1}\ad(h)^{-1}=\ad([g,h])
$$
and
$$
[g,h]w[g,h]^{-1}=w \ {\rm mod }\ F_2''
$$
for any $w\in F_2'$,
we have
$$
\ad(g)\ad(h)=\ad(h)\ad(g)\in \Aut((F'_2)^{\ab}).
$$
Thus the above adjoint action of $F_2$ on $(F'_2)^{\ab}$ induces an adjoint action of $F_2^{\ab}$ on $(F'_2)^{\ab}$.
Therefore, $(F'_2)^{\ab}$ becomes a module over a
(commutative) Laurent polynomial ring $\bZ[\ad(\bar{a})^{\pm 1},\ad(\bar{b})^{\pm 1}]$
generated by two elements
$\ad(\bar{a}),\ad(\bar{b})$.

\begin{lemma}
\label{generation of second abelian}
The group $(F'_2)^{\ab}=F_2'/F_2''$ is generated by $[a,b]$ over $\bZ[\ad(\bar{a})^{\pm 1},\ad(\bar{b})^{\pm 1}]$.
\end{lemma}

\begin{proof}
We have
\begin{align*}
[gh,x]=
ghxh^{-1}g^{-1}x^{-1}=
ghxh^{-1}x^{-1}g^{-1}gxg^{-1}x^{-1}=\ad(g)([h,x])\cdot [g,x]
\end{align*}
and
$[g,h]=[h,g]^{-1}$.
By writing the product of $(F'_2)^{\ab}$ additively, we have
$$
[gh,x]=\ad(g)([h,x])+[g,x] \text{ and } [g,h]=-[h,g].
$$
The lemma follows from these relations.
\end{proof}


\subsection{A group homomorphism from \texorpdfstring{$F_2$}{F2} to \texorpdfstring{$\PGL(2,\bbK)$}{PGL(2,K)}}
\label{field K}

Let $\bbK$ be an algebraically closed field whose transcendental degree over the prime field $\bk$ is greater than
or equal to $2$.
Let $\alpha, \beta$ be elements algebraically independent over $\bk$.
We set
$$
\widetilde{A}=
\left(
\begin{matrix}
\alpha & 1 \\
0 & 1 
\end{matrix}\right),
\quad
\widetilde{B}=
\left(
\begin{matrix}
1 & 0 \\
0 & \beta 
\end{matrix}
\right).
$$
The class of $\widetilde{A}, \widetilde{B}$ in 
$\PGL(2,\bbK)=\Aut(\bP_\bbK^1)$ is denoted by $A, B$.
We define a group homomorphism $$\varphi \colon F_2 \longrightarrow \PGL(2,\bbK) \text{ via } \varphi(a)=A, \varphi(b)=B. 
$$
One can easily check the following lemma.

\begin{lemma}
The image of $F_2'$ under $\varphi$ is contained in
$$
N=\Bigg\{
\left(
\begin{matrix}
1 & n \\
0 & 1 
\end{matrix}\right) : n\in \bbK
\Bigg\}.
$$
Since $N$ is an abelian group,
the image of $F_2''$ under $\varphi$ is the identity group.
Thus we have a homomorphism between abelian groups:
\begin{equation}
\label{second abelianization} 
\varphi' \colon (F_2')^{\ab} = F_2'/F_2'' \longrightarrow N.
\end{equation}
\end{lemma}

We introduce an $R=\bZ[\xi^{\pm 1},\eta^{\pm 1}]$-module structure on the abelian group $N$ by setting
$$
\xi^i\eta^j
\left(
\begin{matrix}
1 & n \\
0 & 1 
\end{matrix}\right) =
\left(
\begin{matrix}
1 & \alpha^i\beta^{-j}n \\
0 & 1 
\end{matrix}\right).
$$
By a direct calculation, we have the following lemma.

\begin{lemma}
Via the identification between
$\bZ[\ad(\bar{a})^{\pm 1},\ad(\bar{b})^{\pm 1}]$
and
$R=\bZ[\xi^{\pm 1},\eta^{\pm 1}]$ given by 
$$
\xi=\ad(\bar{a}),\quad \eta=\ad(\bar{b}),
$$
the homomorphism $\varphi'$ becomes an $R$-homomorphism.
\end{lemma}

Then the image of the commutator $[a,b] \in F_2'$ under 
the homomorphism $\varphi$
is equal to
$$
[A,B]=\left(
\begin{matrix}
1 & 1-\beta^{-1} \\
0 & 1 
\end{matrix}\right).
$$
By Lemma \ref{generation of second abelian},
$(F_2')^{\ab}$ 
is generated by $[a,b]$ 
as
a $\bZ[\ad(\bar{a})^{\pm 1},\ad(\bar{b})^{\pm 1}]$-module.
Therefore, the image of $\varphi'$ in the abelian group $N$ is equal to
\begin{align}
\label{infinitely generated abelian}
&\bZ[\xi^{\pm 1},\eta^{\pm 1}]\left(
\begin{matrix}
1 & 1-\beta^{-1} \\
0 & 1 
\end{matrix}\right)
\\
\nonumber
=&\Bigg\{\left(
\begin{matrix}
1 & n \\
0 & 1 
\end{matrix}\right) :
n\in \bZ[\alpha^{\pm 1},\beta^{\pm 1}](1-\beta^{-1})
\Bigg\},
\end{align}
which is not finitely generated since our $\alpha$ and $\beta$ are assumed to be algebraically independent over $\bk$. 


\subsection{Homological trivial part of a finitely generated solvable subgroup of automorphisms}

Let $\bbK$ be the field as in \S\ref{field K}.
Let $Y$ be a smooth projective $K3$ surface over $\bbK$
such that the automorphism group $\Aut(Y)$ of $Y$ contains a subgroup $M$ isomorphic to $\bZ^2$.
We choose a generator $\{p,q\}$ of $M$.
We set 
$X = \bP_\bbK^1\times Y$.
We define the {\it homologically trivial part $\Aut^0(X)$} of $\Aut(X)$
by 
$$
\Ker(\rho \colon \Aut(X) \lra \GL(H^*(X,\bQ))).
$$ 
Then $\Aut^0(X)$ is isomorphic to $\Aut(\bP_\bbK^1)=\PGL(2,\bbK)$.
Let $G$ be the subgroup of $\Aut(X)$ generated by
$$
s=(A, p),\quad
t=(B, q) \in \Aut(\bP_\bbK^1)\times \Aut(Y),
$$
where $A, B$ are elements in $\PGL(2,\bbK)$ defined in \S\ref{field K}.
Then $G$ becomes a finitely generated solvable subgroup of $\Aut(X)$.
In particular, the solvable length of $G$ is $2$.
\begin{proposition}
The subgroup 
$$
G^0 \coloneqq \Aut^0(X)\cap G = \Ker(G \lra \GL(H^*(X,\bQ)))
$$ of $G$ is not finitely generated.
\end{proposition}
\begin{proof}
We define homomorphisms $\psi \colon F_2 \to G$,
$\pi \colon F_2 \to M$ and $j \colon M \lra \GL(H^*(X))$
by
$$
\psi(a)=s,\quad
\psi(b)=t, \quad \pi(a)=p, \quad \pi(b)=q,
$$
and
$$
j(m)=(\id_{\bP_\bbK^1},m)^*.
$$
Then the morphism $\psi$ is surjective and $j$ is injective.
We consider the diagram
\[\xymatrix{
& F_2 \ar@{->>}[r]^-{\psi} \ar@{->>}[d]^-{\pi} \ar@_{->>}[ld] \ar@^{->>}[ld] & G \ar[d]^-{\rho} 
& \\
F_2/F_2' \ar[r]^-{\simeq} & M 
\ar@{^{(}->}[r]^-{j}
& \GL(H^*(X)).
}\]
Since the actions of $A, B \in \Aut(\bP_\bbK^1)$ on $H^*(\bP_\bbK^1)$ are trivial,
we have $\rho \circ \psi = j \circ \pi$.
The morphism $\pi$ will be identified with the abelianization of $F_2$.

We show that the group $G^0$ is equal to the image $\psi(F_2')$ of $F_2'$
under $\psi$. Let $k$ be an element of $G^0$.
Using the surjectivity of $\psi$, we choose an element 
$\widetilde{k}$ of $F_2$ such that $\psi(\widetilde{k})=k$.
Since $(j\circ \pi)(\widetilde{k})=(\rho\circ \psi)(\widetilde{k})=\rho(k)=\id$,
and $j$ is injective, we have
$$
\widetilde{k} \in \Ker(F_2 \xrightarrow{\ \pi\ } M)=F_2'.
$$
Thus we have $k\in \psi(F_2')$.

We set $G'=[G,G]$. Since $\psi$ is surjective, the group $\psi(F_2')=G^0$ is identified with $G'$.
Thus we have $G'=G^0\subset \Aut(\bP_\bbK^1)$.
Also, the image of $F''_2$ in $G'$ under the morphism $\psi$ is trivial.
The composite homomorphism
$$ (F_2')^{\ab} \longsurjmap G' =G^0 \longinjmap \Aut(\bP_\bbK^1)=\PGL(2,\bbK)$$
is nothing but the homomorphism
$\varphi'$ defined in 
(\ref{second abelianization}).
Therefore, $G^0 = \im(\varphi')$ is equal to the group
(\ref{infinitely generated abelian}) by Lemma 
\ref{generation of second abelian}, which is not
finitely generated.
\end{proof}

\vskip 2em

\end{document}